\documentclass[a4paper,10pt]{preprint}
\usepackage[left=1.3in,right=1.3in,bottom=1.4in]{geometry}  
\usepackage[full]{textcomp}
\usepackage[osf]{newtxtext}

\usepackage{mathalfa}
\usepackage{microtype}

\usepackage{orcidlink}

\usepackage{booktabs}
\usepackage{enumitem}
\usepackage{amsmath}
\usepackage{amsfonts} 
\usepackage{amssymb}             
\usepackage{stmaryrd}
\usepackage{comment}
\usepackage{ulem}
\usepackage{mathrsfs}
\usepackage{booktabs}
\usepackage{mathtools}
\usepackage{tikz}
\usepackage{amsthm}                
\usepackage{mhequ}
\usepackage{hyperref}

\definecolor{labelkey}{gray}{.8}
\definecolor{refkey}{gray}{.8}

\hypersetup{pdfstartview=}
\addtocontents{toc}{\protect\hypersetup{hidelinks}}
\DeclareSymbolFont{sfoperators}{OT1}{ptm}{m}{n}
\DeclareSymbolFontAlphabet{\mathsf}{sfoperators}

\makeatletter
\def\operator@font{\mathgroup\symsfoperators}
\makeatother

\numberwithin{equation}{section}

\newtheorem{thm}{Theorem}[section]
\newtheorem{defn}[thm]{Definition}
\newtheorem{lem}[thm]{Lemma}
\newtheorem{prop}[thm]{Proposition}

\newtheorem{cor}[thm]{Corollary}

\newtheorem{assumption}[thm]{Assumption}
\theoremstyle{remark}

\newtheorem{rmk}[thm]{Remark}

\def\grads{\grad_{\!\mathrm{sym}}}

\def\Wick#1{\mathord{{:}{#1}{:}}}

\makeatletter
\def\th@newremark{\th@remark\thm@headfont{\bfseries}}

\def\bdiamond{\mathop{\mathpalette\bdi@mond\relax}}
\newcommand\bdi@mond[2]{%
	\vcenter{\hbox{\m@th
			\scalebox{\ifx#1\displaystyle 2.6\else1.8\fi}{$#1\diamond$}%
	}}%
}

\def\bDiamond{\mathop{\mathpalette\bDi@mond\relax}}
\newcommand\bDi@mond[2]{%
	\vcenter{\hbox{\m@th
			\scalebox{\ifx#1\displaystyle 2.6\else1.2\fi}{$#1\Diamond$}%
	}}%
}
\makeatother

\newcommand{\EE}{\mathbb{E}}     

\newcommand{\PP}{\mathbb{P}}     

\newcommand{\aA}{\mathcal{A}}
\newcommand{\bB}{\mathcal{B}}
\newcommand{\cC}{\mathcal{C}}

\newcommand{\fF}{\mathcal{F}}

\newcommand{\hH}{\mathcal{H}}

\newcommand{\lL}{\mathcal{L}}
\newcommand{\mM}{\mathcal{M}}

\newcommand{\pP}{\mathcal{P}}

\newcommand{\rR}{\mathcal{R}}
\newcommand{\sS}{\mathcal{S}}

\newcommand{\xX}{\mathcal{X}}

\newcommand{\fA}{\mathfrak{A}}
\newcommand{\fB}{\mathfrak{B}}

\newcommand{\fK}{\mathfrak{K}}

\makeatletter 
\newcommand{\cov}{{\operator@font cov}}
\newcommand{\var}{{\operator@font var}}
\newcommand{\corr}{{\operator@font corr}}
\newcommand{\diam}{{\operator@font diam}}
\newcommand{\Av}{{\operator@font Av}}
\newcommand{\trig}{{\operator@font trig}}
\newcommand{\Enh}{{\operator@font Enh}}
\newcommand{\EEnh}{\overline {\operator@font Enh}}
\makeatother

\let\div\undefined
\DeclareMathOperator{\supp}{supp}
\DeclareMathOperator{\div}{div}

\newcommand{\C}{\mathbf{C}}

\newcommand{\M}{\mathbf{M}}
\renewcommand{\P}{\mathbf{P}}

\renewcommand{\L}{\mathbf{L}}
\newcommand{\N}{\mathbf{N}}

\newcommand{\R}{\mathbf{R}}
\newcommand{\T}{\mathbf{T}}

\newcommand{\Z}{\mathbf{Z}}

\newcommand{\one}{\mathbf{1}}

\renewcommand{\d}{\partial}

\newcommand{\eps}{\varepsilon}

\newcommand{\Laplace}{\Delta}
\newcommand{\dtLaplace}{\d_t-\Laplace}
\newcommand{\md}{\mathrm{d}}
\newcommand{\bPi}{\boldsymbol{\Pi}}
\newcommand{\bxi}{\boldsymbol{\xi}}
\newcommand{\bPhi}{\boldsymbol{\Phi}}

\newcommand{\grad}{\nabla}
\newcommand{\wl}{w^{\lL}}
\newcommand{\bwl}{\Bar{w}^{\lL}}
\newcommand{\wh}{w^{\hH}}
\newcommand{\bwh}{\Bar{w}^{\hH}}

\newcommand{\taut}[1]{\tau_t^{(#1)}}

\newcommand{\ur}{u_{\mathfrak{r}}}
\newcommand{\us}{u_{\mathfrak{s}}}

\newcommand{\Besov}[2]{\bB^{#1}_{#2}}
\DeclarePairedDelimiter{\bracket}{\langle}{\rangle}

\newcommand{\para}{\olessthan}
\newcommand{\rpara}{\ogreaterthan}
\newcommand{\reso}{\odot}

\newcommand{\pareq}{\mathrel{\rlap{\kern0.07em\tikz[x=0.1em,y=0.1em,baseline=0.04em] \draw[line width=0.3pt] (0,1.93) -- (4.5,0);}\olessthan}}
\newcommand{\rpareq}{\mathrel{\rlap{\kern0.23em\tikz[x=0.1em,y=0.1em,baseline=0.04em] \draw[line width=0.3pt] (0,0) -- (4.5,1.93);}\ogreaterthan}}

\def\dash{\leavevmode\unskip\kern0.18em--\penalty\exhyphenpenalty\kern0.18em}
\def\slash{\leavevmode\unskip\kern0.15em/\penalty\exhyphenpenalty\kern0.15em}

\newcommand{\eqlaw}{\stackrel{\mbox{\tiny\rm law}}{=}}

\makeatletter

\DeclareRobustCommand{\TitleEquation}[2]{\texorpdfstring{\StrLeft{\f@series}{1}[\@firstchar]$\if%
b\@firstchar\boldsymbol{#1}\else#1\fi$}{#2}}

\makeatother

\def\smallstep{\par\smallskip\noindent\textit}
\def\step{\paragraph*}

\colorlet{darkblue}{blue!90!black}
\colorlet{darkred}{red!90!black}

\begin{document}
\title{Ergodicity of 2D singular stochastic Navier--Stokes equations}
\author{Martin Hairer$^1$\orcidlink{0000-0002-2141-6561} and Wenhao Zhao$^2$\orcidlink{0000-0002-7489-1884}}
\institute{
Imperial College London, UK and EPFL, Switzerland.
\email{martin.hairer@epfl.ch}
\and
EPFL, Switzerland. \email{wenhao.zhao@epfl.ch}
}

\maketitle

\begin{abstract}
We consider the 2D stochastic Navier--Stokes equations driven by
noise that has the regularity of space-time white noise but doesn't exactly coincide with it.
We show that, provided that the intensity of the noise is sufficiently weak at high frequencies,
this systems admits uniform bounds in time, so that it has an invariant measure for which 
we obtain stretched exponential tail bounds.
\end{abstract}

\setcounter{tocdepth}{2}
\microtypesetup{protrusion=false}
\tableofcontents
\microtypesetup{protrusion=true}

\section{Introduction}
The aim of the present article is to obtain a priori bounds for the following stochastic Navier--Stokes equation
\begin{equ}
\label{e:SNS}
    \d_t u = \Laplace u - \P \div(u\otimes u) + \xi\;,
\end{equ}
where $u:\R_+\times \T^2 \rightarrow \R^2$ is the velocity, $\P$ is the Leray projection and 
the noise $\xi$ satisfies Assumption~\ref{as:noise} below, which has the same regularity as 
space time white noise. 

Our main motivation is to develop new techniques to obtain a priori 
bounds to singular SPDEs, in particular those without a strong damping term such as $\Phi^4$. 
A prime example is the stochastic Yang--Mills equation studied in \cite{CCHS22, CCHS24}, where 
a uniform in time bound on the solution would lead to a PDE-based construction of 
the (still putative in three dimensions) continuum Yang--Mills measure.
Equation \eqref{e:SNS} shares several features with 2D stochastic Yang--Mills, 
including a logarithm divergence in renormalisation, a nonlinear term of the form $u \grad u$, 
and the lack of strong damping effects. However, a crucial difference lies in the conserved 
quantities: the $L^2$ norm of the velocity field is decreasing in Navier--Stokes, whereas for 
Yang--Mills it is the $L^2$ norm of the curvature, corresponding to a type of $H^1$ norm. 
This makes stochastic Navier--Stokes more 
tractable than the stochastic Yang--Mills, as the solution of \eqref{e:SNS} lies in 
$\cC^{-\kappa}$ for any $\kappa > 0$, which is ``almost'' in $L^2$.
As such, we view \eqref{e:SNS} as a nice toy model (as opposed to a realistic model
for turbulence, which would rather use smooth noise of the type considered
in \cite{HM06}), which we hope will provide insights 
useful to the study of 2D stochastic Yang--Mills.

Our assumption on the driving noise goes as follows.
\begin{assumption}
    \label{as:noise}
    
Let $e_k(x):= e^{2\pi ik\cdot x} \frac{k^{\perp}}{|k^{\perp}|}$, where $k^{\perp} = (k_2, -k_1)$.
Construct a probability space $(\Omega, \PP)$ with complex Brownian Motions $\{B_k\}_{k\in \Z^2}$ such that
\begin{equ}
    \EE B_k(t) = \EE B_k^2(t) = 0\;, \quad \quad
    \EE |B_k(t)|^2 = t\;, \quad \quad B_k = \Bar{B}_{-k}
\end{equ}
and $B_k$ and $B_l$ are independent if $k\neq \pm l$. 
Then we impose $\xi$ to be of the form
\begin{equ}
    \label{e:xi_form}
    \xi = \sum_{k\neq 0} \phi_k \md B_k(t) \cdot e_k\;,
\end{equ}
where $\phi_k\in \C$ satisfies that $\phi_k = \Bar{\phi}_{-k}$, and
$    \limsup_{k\rightarrow \infty} |\phi_k| < \infty\,$.
\end{assumption}
The case where the forcing term is given by space-time white noise (i.e.\ $\phi_k = 1$) was
considered in \cite{DPD02} where global in time solutions were constructed. The approach in \cite{DPD02} however relied crucially on the 
fact that, in this particular case, the invariant measure for \eqref{e:SNS} is known explicitly
(it is simply the Gaussian measure that is invariant for the corresponding stochastic heat equation).
As soon as the intensity of even one single Fourier mode of the driving noise is slightly perturbed,
this technique fails. 

In \cite{HR24}, the first author and Rosati considered the case of noises that are smooth perturbations
of space-time white noise. Using pathwise techniques, they obtained global well-posedness, but with
an a priori bound that is roughly of the form
\begin{equ}
	\d_t \|u\| \leq C \|u\|\log\|u\|\;,
\end{equ}
yielding a double exponential growth in time. (In fact it is even worse since the constant $C$
 depends slightly on the time horizon under consideration.) Loosely speaking, the main result in the present article is a uniform in time bound for noise $\xi$ satisfying Assumption~\ref{as:noise} with $\limsup_{k\rightarrow \infty} |\phi_k|$ sufficiently small (but possibly non-zero). This allows one to prove ergodicity for the equation, under an additional mild non-degeneracy assumption on the noise. Ergodicity of stochastic Navier--Stokes equations with noise smoother than white noise has been studied in various works (see e.g. \cite{FG95, FM95, GM05, HM06} etc.). To the best of our knowledge, our result is the first time that ergodicity of the stochastic Navier--Stokes equation is proved, with a noise as rough as the space-time white noise.
  
Before stating our result, we review some recent progress on a priori bounds for singular stochastic PDEs. For $\Phi^4$ equations, due to the existence of a strong damping term, the long time behaviour is now well understood, see e.g.\ \cite{MW17b, GH19, MW20}. Uniform in time bounds were proved, which allowed to essentially complete
Parisi and Wu's stochastic quantisation programme \cite{ParisiWu} in the particular case of the $\Phi^4_3$ measure. Besides constructing the measure, it is also possible to use stochastic quantisation to study its properties. 
For example, in \cite{HS22b} a quartic tail bound is proved for $\Phi^4_3$ using stochastic quantisation, which is the best currently known tail bound for $\Phi^4_3$.

A priori bounds for singular SPDEs without strong damping term turns out to be much more difficult. For the KPZ equation without using the Hopf--Cole transform, the global results obtained up to now rely on its relation to Hamilton--Jacobi equations, see \cite{GP17} and \cite{ZZZ22} for example. Recently in \cite{CFW24,SZZ24} a priori bounds for the generalised Parabolic Anderson Model (gPAM) have been studied, strongly making use of the maximum principle. A construction for sine-Gordon quantum field theory slightly beyond the first threshold is also obtained in \cite{CFW24}. In \cite{BC24b,BC24a}, the global well-posedness for the 2D stochastic Abelian Higgs equation and the sine-Gordon model up to the third threshold are obtained. One of the main inputs in these two papers is the introduction of modified stochastic objects taking the initial data into account, while the probabilistic bounds for these objects exhibit rather weak dependence on the initial data. Applying this idea requires a good analysis on the linearised equation, see the discussions in \cite[Section~1.4]{BC24a}.

From this discussion, we can already see that obtaining a priori bounds for singular SPDEs heavily relies on the exact form of the equation, as is also often the case in the 
deterministic case.
In the cases of KPZ, gPAM and sine-Gordon, a more than logarithmic divergence in the renormalisation is allowed. However, in the present treatment of the 2D Navier--Stokes
equations, the fact that we are at the ``borderline'' of classical well-posedness 
(the term $u\otimes u$ in \eqref{e:SNS} requires renormalisation by a logarithmically 
divergent constant, although this constant is eventually ``killed'' when taking the divergence) is still necessary. Compared to many other models, the 2D Navier--Stokes 
equations possess the difficulty of being a vector-valued equation,
so that the maximum principle is not available to us and we have to rely on energy estimates to get global well-posedness.

Besides the fact that the energy of our solutions is almost surely infinite for any positive time, 
another major difficulty is the criticality of the $L^2$ norm. As a consequence, the $L^2$ norm of the solution 
cannot be expected to be used to control any stronger norm (for example $\|\cdot\|_{L^{2+}}$ or $\|\cdot\|_{H^{0+}}$ norms) in short time by standard parabolic regularity estimates. On the 2D torus we know of course that enstrophy, namely the
$H^1$ norm, is another conserved quantity for Navier--Stokes, but at this moment we don't know how to exploit 
this since the regularity of our solutions is too far below $H^1$ (as already pointed out,
the same difficulty arises in the case of 2D stochastic Yang–Mills). Moreover, the linearisation of the Navier--Stokes equations doesn't behave as nice as the Abelian Higgs and the sine-Gordon equation, so it remains unclear how to apply the idea in \cite{BC24b, BC24a} to our setting at the moment.

Our main novelty in improving \cite{HR24} is the introduction of a new ansatz. The philosophy for getting the ansatz is to try to make the rough part in the ansatz as small as possible, so that the equation for the smooth part can be considered as a perturbation of deterministic Navier--Stokes. This suggests that even in the first step, when we perform the Da Prato--Debussche trick, we should subtract only the high frequency part of the linear solution instead of the whole linear solution. This modification forces us to use It\^o's formula when doing energy estimates so that, unlike the previous works on a priori bounds for singular SPDEs that separates probability and analysis at the very beginning, we benefit from taking expectations of certain martingale terms. The idea is quite simple and it does work better in this specific example of Navier--Stokes, but we hope that it could also be applied to other equations. We will discuss our strategy in more detail in Section~\ref{sec:Ansatz}.


\subsection{Main results}
\label{sec:main_results}

Throughout the article we use the convention $\|\cdot\| = \|\cdot\|_{L^2}$, and we fix $\kappa\leq \frac{1}{10^4}$ to be a small positive constant. Since solutions to \eqref{e:SNS} belong to $\cC^{-\kappa}$ and $L^2$ estimates cannot be applied directly, we would like to decompose them into a small $\cC^{-\kappa}$ part and a large $L^2$ part. This motivates us to define the quantity
\begin{equ}
\label{e:Lyapunov_func}
    V_{\alpha}(u) := \inf \{ \|\us\| + \|\ur\|_{\cC^{-\kappa}}: u = \us + \ur, \|\ur\|_{\cC^{-\kappa}} \leq \alpha\}\;.
\end{equ}
Our main result is the following uniform in time bound, the proof of which is given at the very end of Section~\ref{sec:tail_bound}. 

\begin{thm}
\label{thm:main_result}
Let $P_t$ be the Markov semigroup of \eqref{e:SNS}. Let $V_{\alpha}: \cC^{-\kappa} \rightarrow \R$ defined above be the Lyapunov function for the Markov process. There exist some positive constants $0 < \alpha_0\leq 1$ and $C > 0$ such that, if the noise $\xi$ satisfies Assumption~\ref{as:noise} and $\limsup_{k\rightarrow \infty} |\phi_k| \leq \frac{1}{2}\alpha_0$\,, then there exists some $\gamma > 0$ such that for any $t \geq 1$ and $x\in \cC^{-\kappa}$, we have
\begin{equ}
\label{e:main_result1}
    \big(P_{t} V_{2\alpha_0}\big)(x) \leq C e^{-\gamma t} V_{2 \alpha_0}(x) + C\;.
\end{equ}
As a result, the solution to \eqref{e:SNS} admits at least one invariant measure $\mu_\star$. 
Moreover, for any $N \in \N^*$ there exists some $\gamma_N > 0$ such that for any $x \in L^2 \cap \cC^{-\kappa}$ and $t \geq 0$ we have 
\begin{equ}
\label{e:main_result2}
    \big(P_{t} V^N_{2\alpha_0}\big)(x) \leq C^N e^{-\gamma_N t} \|x\|^N + (CN)^{2N}\;.
\end{equ}
\end{thm}

\begin{rmk}
	Here in \eqref{e:main_result1} the result is only proved for $t$ away from $0$, since our main goal is to study the long time behaviour of the solution. This allows us to avoid some technical problems caused by the rough part of the initial data, as it is not smoothened by the heat flow in short time. 
\end{rmk}
\begin{rmk}
\label{rmk:large_alpha}
	Our main estimate, Proposition~\ref{pr:main_bound}, actually does not require the smallness of $\alpha_0$. It is only used in Proposition~\ref{pr:decay_1} in order to get \eqref{e:decay_1}, which can be viewed as a stopping time version of \eqref{e:main_result1}. 
If $\alpha_0$ is large, the estimate in Proposition~\ref{pr:main_bound} should still yield a 
bound of the form \eqref{e:main_result1}, but with a possibly negative exponent $\gamma$. 
It is unclear to us whether this hints at a genuine transition between existence and
non-existence of an invariant measure at large values of $\alpha_0$ or, as is more likely,
whether it just reflects a limitation of our current proof technique. 
\end{rmk}

If the driving noise is non-degenerate, the uniqueness of the invariant measure follows from the strong Feller property and full support of the equation. The strong Feller property for singular SPDEs was studied in \cite{HM18} and  extended in \cite{ZZ17} to the 2D stochastic Navier--Stokes equations with space-time white noise. Although our noise is not exactly space-time white noise, the strong Feller property does hold under a natural non-degeneracy condition following the exact same argument. 
(We also provide a short proof in the last section since \cite{ZZ17} mainly focuses on the 3D case.) The support theorem is well studied in \cite{HS22a} for a very general class of singular SPDEs. In \cite{CF18} and \cite{TW18} simpler cases such as gPAM and $\Phi^4_2$ are treated. Since no renormalization constant appears in \eqref{e:SNS}, the support theorem can be proven relatively easily. Combining these with the decay of Lyapunov function in \eqref{e:main_result1}, we have the following exponential mixing result.
\begin{thm}
\label{thm:exp_mixing}
    If the noise satisfies Assumption~\ref{as:noise} with $\limsup_{k\rightarrow \infty}|\phi_k| \leq \frac{1}{2}\alpha_0$ and
    \begin{equ}
    \label{e:noise_irreducible}
        \inf_{k\neq 0} |\phi_k| > 0\;,
    \end{equ}
    then \eqref{e:SNS} has exactly one invariant measure $\mu_\star$ on $\cC^{-\kappa}$. Moreover, there exist some constants $C, \Tilde{\gamma}, T_* > 0$ such that 
    \begin{equ}
    \label{e:exp_mixing}
    	\|\pP_t(x, \cdot) - \mu_\star \|_{TV} \leq C\big(1 + V_{2\alpha_0}(x)\big)e^{-\Tilde{\gamma} t}
    \end{equ}
    for any $t \geq T_*$ and $x\in \cC^{-\kappa}$, and $\mu_\star$ satisfies the tail bound
\begin{equ}
\label{e:exp_tail}
	\mu_\star(V_{2\alpha_0}(x) \geq K) \leq C\exp(-C^{-1}\sqrt{K})\;.
\end{equ}

    \end{thm}
\begin{rmk}
	Condition~\eqref{e:noise_irreducible} is a technical assumption guaranteeing that the noise is not degenerate. See \cite{HM06} for the case of a smooth but degenerate noise, where the existence of an invariant measure is easy to prove but the question of its uniqueness is highly nontrivial.
\end{rmk}

\subsection{Notations}
We try to keep notations as close as possible to \cite{HR24}. We identify $\M^d$, the space of $d\times d$ matrices with $\R^d \otimes \R^d$ in the usual way. Denote $u\otimes_s v = \frac{1}{2}(u \otimes v + v\otimes u)$. For $\varphi \in C^1(\T^2; \M^2)$ and $j = 1,2$, we set
\begin{equ}
	\div(\varphi)_j (x) = \bigl(\div(\varphi)(x)\bigr)_j = \sum_{i=1}^{2} \d_i \varphi_{i,j}(x) \in C(\T^2; \R)\;.
\end{equ}
For $\varphi \in C^1(\T^2; \R^2)$, we define $\grad \varphi, \grads \varphi \in C(\T^2; \M^2)$ by
\begin{equ}
	(\grad \varphi)_{i,j} = \d_i \varphi_j \;, \qquad (\grads\varphi)_{i,j} = \frac{1}{2} (\d_i \varphi_j + \d_j \varphi_i)\;.
\end{equ}
Here $C$ and $C^1$ denote the spaces of continuous and continuously differentiable functions, respectively. We use the convention $\|\cdot\| = \|\cdot\|_{L^2}$ throughout the paper. We use $A\lesssim B$ to denote that there exists some constant $C$ such that $A\leq C B$. We write $A\lesssim_{\kappa} B$ to emphasise that the constant may depend on $\kappa$. The constants $C$ may vary from line to line to simplify notations. We use $a \wedge b$ and $a\vee b$ to denote $\min\{a, b\}$ and $\max\{a, b\}$ respectively.

Denote the space of Schwartz distributions by $\sS'(\T^2; \R^d)$. For every $\varphi \in \sS'(\T^2; \R^d)$ and $k\in \Z^2$, define the Fourier transform to be 
\begin{equ}
	 \hat{\varphi}(k) = \fF\varphi(k) := \int _{\T^2} e^{-2\pi ik\cdot x} \varphi(x) \md x\;.
\end{equ}
Denote the space of mean-free Schwartz distributions by
\begin{equ}
	\sS'_{\times}(\T^2; \R^d) := \{\varphi \in \sS'(\T^2; \R^d): \hat{\varphi}(0) = 0\}\;.
\end{equ}
Take a partition of unity $\chi$ and $\rho_j = \rho(2^{-j} \cdot)$ as in \cite[Proposition~2.10]{BCD11} with the convention $\rho_{-1} = \chi$. 
For any $j \geq -1$, define the Littlewood--Paley projection 
\begin{equ}
\label{e:Littlewood_Paley}
	\Delta_j f = \fF^{-1}(\rho_j \fF f)\;.
\end{equ}
For $j \in \R_{+}$, we use the convention $\Delta_j f := \Delta_{\lceil j \rceil} f$ and, for 
$1\leq p \leq \infty$ and any $\varphi: \T^2 \rightarrow \R^d$ we write $\|\varphi\|_{L^p}$
for the usual $L^p$ norms.
Then for $\alpha \in \R$, $1\leq p, q \leq \infty$, we define the mean-free Besov spaces 
$\Besov{\alpha}{p,q}(\T^2; \R^d) \subset \sS'_{\times}(\T^2; \R^d)$ via norms
\begin{equ}[e:BesovNorm]
	\|\varphi\|_{\Besov{\alpha}{p,q}} := \Big(\sum_{j \geq -1} 2^{j\alpha q} \|\Delta_j \varphi \|_{L^p(\T^2; \R^d)}^q \Big)^{\frac{1}{q}}\;.
\end{equ}
We distinguish the case $p = q = 2$ by denoting 
\begin{equ}
	H^{\alpha}(\T^2; \R^d) = \Besov{\alpha}{2,2}(\T^2; \R^d)
\end{equ}
as well as the case $p = q = \infty$ by denoting 
\begin{equ}
	\cC^{\alpha}(\T^2; \R^d) = \Besov{\alpha}{\infty,\infty}(\T^2; \R^d)\;.
\end{equ}
Note that $\cC^0$ is not $L^{\infty}$ in this notation.
For any functional space $X$, we use $X_{\div}$ to denote the subspace of $X$ that consists of divergence free functions in $X$ (e.g. $\cC^{\alpha}_{\div}$\, , $L^2_{\div}$).

Given a time dependent measurable function $\varphi: [0,t] \rightarrow X$ for some Banach space $X$, define 
\begin{equ}
	\|\varphi\|_{L^p_t X} := \left(\int_{0}^{t} \|\varphi(s)\|_X^p \md s\right)^{\frac{1}{p}}\;.
\end{equ}
For a space-time function $\varphi: [0,t] \times \T^2 \rightarrow \R^d$, we use $\varphi[s]$ to denote the function $\varphi(s, \cdot)$. We let $(\dtLaplace)^{-1} \varphi$ be the unique solution $u: [0,t] \times \T^2 \rightarrow \R^d$ of the inhomogeneous heat equation
\begin{equ}
	(\dtLaplace) u = \varphi \,, \qquad u[0] = 0. 
\end{equ}  

\subsubsection*{Organisation of the article}

The rest of the article is organized as follows. In Section~\ref{sec:Ansatz} we discuss our strategy for the proof. In Section~\ref{sec:bound} we prove \eqref{e:main_result1} in Theorem~\ref{thm:main_result} to illustrate the main idea. In Section~\ref{sec:tail_bound} we sharpen the methods to prove \eqref{e:main_result2}. In Section~\ref{sec:exponential_mixing} we prove the exponential mixing result Theorem~\ref{thm:exp_mixing}.

\subsubsection*{Acknowledgements}

We are grateful to Bjoern Bringmann for having spotted an error in an earlier version
of the bound \eqref{e:large_freq1}.

\section{Preliminaries and strategies}
\label{sec:Ansatz}
\subsection{Preliminaries}
\begin{defn}
\label{def:paraproduct}
	For $\varphi, \psi \in \sS'(\T^2; \R^d)\,$, we define the high-low paraproduct
	\begin{equ}
		\varphi \para \psi := \sum_{j = -1}^{\infty} \sum_{i=-1}^{j-2}\Delta_i \varphi \otimes_s \Delta_j \psi\;, \qquad \varphi \rpara \psi := \psi \para \varphi\;.
	\end{equ}
	We also define the high-high paraproduct (also called resonant term) by 
	\begin{equ}
		\varphi \reso \psi = \sum_{|i-j|\leq 1} \Delta_i \varphi \otimes_s \Delta_j \psi\;.
	\end{equ}
	Sometimes we use $\varphi \pareq \psi := \varphi \para \psi + \varphi \reso \psi$ and $\varphi \rpareq \psi := \varphi \rpara \psi + \varphi \reso \psi$\,.
\end{defn}
The following estimates for paraproducts will be used repeatedly, see \cite[Theorems~2.82 and 2.85]{BCD11}.
\begin{lem}
\label{le:paraproduct}
	For $\alpha, \beta \in \R$ and $p, p_1, p_2, q, q_1, q_2 \in [1, \infty]$ such that 
	\begin{equ}
		\frac{1}{p} := \frac{1}{p_1} + \frac{1}{p_2}\;, \qquad \frac{1}{q} := \frac{1}{q_1} + \frac{1}{q_2}\;,
	\end{equ}
	the following bounds hold uniformly in $\varphi, \psi$.
	\begin{equs}
	\label{e:para1}
		\| \varphi \para \psi \|_{\Besov{\alpha}{p, q}} & \lesssim \| \varphi
    \|_{L^{p_1}} \| \psi \|_{\Besov{\alpha}{p_2, q}} \;, & &  \\ 
    \label{e:para2}
    \| \varphi \para \psi \|_{\Besov{\alpha+ \beta}{p, q}} &\lesssim \| \varphi
    \|_{\Besov{\alpha}{p_1, q_1}} \| \psi \|_{\Besov{\beta}{p_2, q_2}} \;, \ \ & &  \text{if} \ \ \alpha  <
    0 \;, \\
    \label{e:reso}
    \| \varphi \reso \psi \|_{\Besov{\alpha+ \beta}{p, q}} &\lesssim \| \varphi
    \|_{\Besov{\alpha}{p_1, q_1}} \| \psi \|_{\Besov{\beta}{p_2, q_2}}\;, \ \ & & \text{if} \ \ \alpha
    {+} \beta > 0 \;.
	\end{equs}
	We will call the first two inequalities high-low paraproduct estimate, and the third inequality resonance estimate.
\end{lem}
The following smoothing effect of the heat flow will also be used repeatedly.
\begin{lem}
	Let $\aA = \{x\in \R^2: \frac{3}{4} \leq |x| \leq \frac{8}{3}\}$ be an annulus. Then there exist constants $c, C>0$ such that for any $p \in [1,\infty]$ and $t > 0$ and $\lambda \ge 1$, we have, for all $u$ with $\supp \hat{u} \subset \lambda \aA$, the bounds
	\begin{equs}
	\label{e:heat_flow1}
		  \|e^{t\Laplace} u\|_{L^p} &\leq Ce^{-ct\lambda^2}\|u\|_{L^p}\;;\\
	\label{e:heat_flow3}
		\|e^{t\Delta}u - u \|_{L^p} &\leq C t\lambda^{2} \|u\|_{L^p}\;.
	\end{equs}
	As a result, for any $\alpha, \beta \in \R$ and $1\leq p, q\leq \infty$\,, we have 
	\begin{equ}
	\label{e:heat_flow2}
		\|e^{t\Delta} u\|_{\Besov{\alpha + \beta}{p,q}} \lesssim_{\beta} t^{-\beta/2} \|u\|_{\Besov{\alpha}{p,q}}\;.
	\end{equ}
\end{lem}
\begin{proof}
	Without loss of generality we assume $u$ is scalar-valued. 
	The bounds~\eqref{e:heat_flow1} and \eqref{e:heat_flow3} are the content of \cite[Lemma~2.10 and 2.11]{MW17a}. However, the bound there is on full space $\R^2$ instead of $\T^2$. We briefly discuss how to fill this gap. We only present the proof for \eqref{e:heat_flow3}.
	First, by Young's convolution inequality and the fact that the Fourier transform of $u$ is
	supported in the ball of radius $\lambda$, we have
	\begin{equ}
		\|e^{t\Delta} u - u\|_{L^p} = \| \psi_{t, \lambda} *u\|_{L^p} \leq \|\psi_{t, \lambda}\|_{L^1}\|u\|_{L^p}\;,
	\end{equ}
	where $\psi$ is the function with Fourier transform $\Hat{\psi}_{t, \lambda}(k) = \phi(k/\lambda)\bigl(e^{-t|k|^2} - 1\bigr)$ and $\phi$ is some symmetric smooth function with compact support, such that $\phi(x) = 1$ if $|x| \leq 1$. 
	Let $g_{t, \lambda}: \R^2 \rightarrow \R$ be a function with (continuous) Fourier transform 
	\begin{equ}
		\Hat{g}_{t, \lambda}(\xi) = \phi(\xi / \lambda)\bigl(e^{-t|\xi|^2} - 1\bigr)\;,
	\end{equ}
	then $\psi_{t, \lambda}$ is the periodisation of $g_{t, \lambda}$, i.e. we have 
	\begin{equ}
		\psi_{t, \lambda}(x) = \sum_{n \in \Z^2} g_{t, \lambda} (x+n)\;.
	\end{equ}
	A slight modification of \cite[Proposition A.3]{MW17a} gives that $\|g_{t, \lambda}\|_{L^1(\R^2)} \leq C t\lambda^2$. Then we get
	\begin{equ}
		\|\psi_{t,\lambda}\|_{L^1(\T^2)} \leq \sum_{n \in \Z^2} \int_{\T^2} |g_{t, \lambda}(x + n)| \md x \leq C t \lambda^2\;,
	\end{equ} 
	and the result follows.
	As for \eqref{e:heat_flow2}, by \eqref{e:heat_flow1} we get
	\begin{equ}
		\|\Delta_j e^{t\Delta}u\|_{L^p} = \|e^{t\Delta} \Delta_j u\|_{L^p} \lesssim e^{-ct 2^{2j}} \|\Delta_j u\|_{L^p} \lesssim_{\beta} t^{-\beta/2} 2^{-\beta j} \|\Delta_j u\|_{L^p}\;.
	\end{equ}
	Then \eqref{e:heat_flow2} follows from the definition \eqref{e:BesovNorm} of the Besov norms. 
\end{proof}
An immediate corollary is the following maximal regularity estimate.
\begin{prop}
    There exists some constant $C$, such that if $u$ solves 
    \begin{equ}
        \d_t u = \Laplace u + \varphi\;, \quad u[0] =0\;,
    \end{equ}
    then for any $p \in [1,\infty]$, $q \in [1, \infty)$, $s\in \R$ and $0 < T \leq 1$ we have
    \begin{equs}
    \label{e:Schauder1}
        \sup_{t\in [0,T]}\|u[t]\|_{\Besov{s+2}{p,\infty}} &\leq C \sup_{t\in [0,T]}\|\varphi[t]\|_{\Besov{s}{p,\infty}}\;.\\
    \label{e:Schauder2}
        \int_0^T \|u[t]\|_{\Besov{s+2}{p, q}}^q \md t &\leq C^q \int_0^T \|\varphi[t]\|_{\Besov{s}{p, q}}^q \md t \;.
    \end{equs}
\end{prop}
\begin{proof}
    Combining Duhamel's formula $u[t] = \int_0^t e^{(t-s)\Laplace}\varphi[s]ds$ with \eqref{e:heat_flow1}, we have
    \begin{equ}
        \|\Delta_i u[t]\|_{L^p} \lesssim \sup_{s\in [0,t]} \|\Delta_i \varphi[s]\|_{L^p}\int_0^t e^{-c(t-s)2^{2i}}ds  \lesssim \frac{1}{2^{2i}} \sup_{s\in [0,t]} \|\Delta_i \varphi[s]\|_{L^p}\;,
    \end{equ}
and \eqref{e:Schauder1} follows from the definition of Besov norms \eqref{e:BesovNorm}. For \eqref{e:Schauder2}, we use \eqref{e:heat_flow1} to get 
\begin{equs}
	\int_0^T \|u[t]\|_{\Besov{s+2}{p, q}}^q \md t &= \sum_{j \geq -1} 2^{(s+2)jq} \int_0^T \|\Delta_j u[t]\|_{L^p}^q \md t \\
	&\leq C^q \sum_{j \geq -1} 2^{(s+2)jq} \int_0^T \Big(\int_0^t e^{-c(t-r)2^{2j}}\|\Delta_j \varphi[r]\|_{L^p} \md r \Big)^q \md t \;.
\end{equs}
By H\"older's inequality, we have 
\begin{equ}
	\Big(\int_0^t e^{-c(t-r)2^{2j}}\|\Delta_j \varphi[r]\|_{L^p} \md r \Big)^q \leq C^q 2^{-2j (q-1)} \int_0^t e^{-c(t-r)2^{2j}}\|\Delta_j \varphi[r]\|_{L^p}^q \md r\;.
\end{equ}
Therefore, by Fubini we get 
\begin{equs}
	\int_0^T \|u[t]\|_{\Besov{s+2}{p, q}}^q \md t &\leq C^q \sum_{j \geq -1} 2^{(s+2)jq}\cdot 2^{-2j (q-1)} \int_0^T \Big(\int_r^T e^{-c(t-r)2^{2j}} \md t \Big) \|\Delta_j \varphi[r]\|_{L^p}^q \md r \\
	&\leq C^q \sum_{j \geq -1} 2^{sjq} \int_0^T \|\Delta_j \varphi[r]\|_{L^p}^q \md r = C^q \int_0^T \|\varphi[r]\|_{\Besov{s}{p, q}}^q \md r\;.
\end{equs}
The result is proved.
\end{proof}
We define the high-low frequency projections as follows.
\begin{defn}\label{def:Hlambda}
	For any $\lambda > 0$, set $\hH_{\lambda,i} = \Delta_{i + \log^{+}_{2}(\lambda)}$ and define the projections 
	\begin{equs}
		\hH_{\lambda}: \sS'(\T^2; \R^2) &\rightarrow \sS'(\T^2; \R^2)\;, &\quad \lL_{\lambda}: \sS'(\T^2; \R^2) &\rightarrow \sS(\T^2; \R^2)\;,\\
		f & \mapsto  \sum_{i \geq 0} \hH_{\lambda,i} f \;, & f &\mapsto f - \hH_{\lambda} f\;.
	\end{equs}
	We write 
	\begin{equ}
		\pP_{\lambda, K} = \one_{\lambda \le K} (\hH_\lambda - \hH_K)
	\end{equ}
	to denote the frequency part between $\lambda$ and $K$. 
\end{defn}

The bounds on Besov norms for these projections follow from the following bound for Littlewood--Paley projection from \cite[Lemma~2.1]{BCD11}.
\begin{lem}
	There exists a constant $C$ such that for any $p \in [1,\infty]$\,, $\alpha, \beta \in \R$\,, $j \geq 0$ and $\varphi \in \sS'(\T^2; \R^2)\,$, we have
	\begin{equ}	\label{e:LP_estimate}
		\|\Delta_j \varphi \|_{\Besov{\beta}{p,\infty}} \leq C^{\beta - \alpha + 1} 2^{j(\beta - \alpha)} \|\Delta_j \varphi \|_{\Besov{\alpha}{p,\infty}}\;.
	\end{equ}
\end{lem}

The following Besov embedding will also be used frequently, see \cite[Proposition~2.71]{BCD11}
\begin{lem}
	Let $1 \leq p_1 \leq p_2 \leq \infty$ and $1 \leq r_1 \leq r_2 \leq \infty.$ Then, for any $s\in \R$, we have $\Besov{s}{p_1, r_1} \hookrightarrow \Besov{s - d\big(\frac{1}{p_1} - \frac{1}{p_2}\big)}{p_2, r_2}$. Also, we have $\Besov{0}{p_1, 1} \hookrightarrow L^{p_1} \hookrightarrow \Besov{0}{p_1, \infty}$.
\end{lem}

\subsection{It\^o's formula}
In this subsection, we revisit the ideas in \cite{HR24} quickly and then we discuss our strategy. There are mainly two inputs, a trick using It\^o's formula and the introduction of a second frequency scale that leads to a new ansatz. We will discuss the first trick in this subsection and the ansatz in the next subsection.

If the noise $\xi$ satisfies Assumption~\ref{as:noise} with $\limsup_{k\rightarrow \infty} |\phi_k| \leq \frac{1}{2} \alpha_0$\,, then $\xi$ can be decomposed into two independent noises $\xi = \xi_1 + \xi_2$\,, where $\xi_1$ satisfies Assumption~\ref{as:noise} with $\sup_{k\neq 0} |\phi_k| \leq \alpha_0$ and $\xi_2$ is a smooth (in space) noise with only finitely many non-zero Fourier modes. From now on we always assume that $\xi = \xi_1 + \xi_2$ with this property.

In general, \eqref{e:SNS} is a singular SPDE and we use the ``Da Prato--Debussche trick'' to define 
a suitable notion of solution. 
Let $X, v$ be solutions of
\begin{equs}
\label{e:linear_solution}
    \d_t X &= \Laplace X + \xi_1\;, \quad X[0]=0\;,\\
\label{e:v_equ}
    \d_t v &= \Laplace v - \P\div(v^{\otimes2} + 2 v \otimes_s X + \Wick{X^{\otimes2}}) + \xi_2\;,\quad v[0] = u[0]\;,
\end{equs}
where $\Wick{X^{\otimes2}}$ is defined by Wick renormalization. We then define solutions to \eqref{e:SNS} by $u := v+X$
and we would like to perform $L^2$ energy estimates on some remainder after subtracting a suitable irregular part. Note that the solution will immediately be distribution-valued, which essentially forces us to allow for initial conditions in $\cC^{-\kappa}$ for some small $\kappa > 0$ that do not have finite energy. We choose to decompose the initial data $u[0]$ into $\ur[0] + \us[0]$, where $\|\ur[0]\|_{\cC^{-\kappa}} \leq 2\alpha_0$ and $\|\us[0]\| \leq V_{2\alpha_0}(u)$ with $V_{2\alpha_0}$ defined in \eqref{e:Lyapunov_func}, and we incorporate the rough part $\ur$ into the 
solution to the linear equation. Precisely, we let $\Tilde{X}$ solve
\begin{equ}
\label{e:X_tilde_equ}
    (\d_t -\Laplace) \Tilde{X} = \xi_1\;, \quad\quad \Tilde{X}[0] = \ur[0]\;,
\end{equ}
and let $Y$ solve
\begin{equ}
\label{e:Y_equ}
    \d_t Y = \Laplace Y - \P \div\big(2Y\otimes_s \Tilde{X} + \Wick{\Tilde{X}^{\otimes 2}}\big)\;, \quad \quad Y[0] = 0\;.
\end{equ}
Note that $\Tilde{X} = X + e^{t\Delta}\ur[0]\,$, and the Wick renormalization is performed for 
$X^{\otimes 2}$, so distributing initial data in this way will not change the renormalization.
The remainder $w := u - \Tilde{X} - Y$ solves 
\begin{equ}
    \d_t w = \Laplace w - \P \div\big(w^{\otimes2}+ 2w\otimes_s(\Tilde{X}+Y) + Y^{\otimes2}\big) + \xi_2\;, \quad \quad w[0] = \us[0] \;.
\end{equ}

Even though $w$ is ``almost'' Lipschitz, the pairing $\langle w, \P\div(w\otimes_s \Tilde{X}) \rangle$ is still
ill-defined if we want to directly calculate $\frac{\md}{\md t}\|w\|^2$, since $w\in \cC^{1-\kappa}$ and $\Tilde{X} \in \cC^{-\kappa}$. Following the idea in \cite{HR24}, it is then natural to subtract from $w$ the most irregular part 
\begin{equ}
\label{e:phi_lambda}
    \phi_{\lambda} := -(\d_t-\Laplace)^{-1} \big(\P\div(2w\para \hH_{\lambda} \Tilde{X}\big)\;,
\end{equ}
where $\lambda$ is some parameter to be chosen later, depending on the size of the initial data. After subtraction, the equation for the remainder $\hat w := w - \phi_\lambda$ includes two terms affecting the choice of $\lambda$, namely
\begin{equ}
\label{e:main_terms}
	\P\div(\hat w \otimes_s \lL_{\lambda} X) \qquad \text{and} \qquad \P\div(\hat w \otimes_s \phi_{\lambda})\;.
\end{equ}
Trying to do a ``naïve'' energy estimate for the remainder $\hat w$, one gets something of the type
\begin{equ}
\label{e:heuristic_1}
	\d_t \|\hat w\|^2 \lesssim -\|\grad \hat w\|^2 + \|\grad \hat w\| \|\hat w\|_{L^p} \|\lL_{\lambda} X\|_{L^{\frac{2p}{p-2}}} + \|\grad \hat w\| \frac{\|\hat w \|_{L^4} \|w\|_{L^4} \|\Tilde{X}\|_{\cC^{-\kappa}}}{\lambda^{1-\kappa}}\;,
\end{equ}
where the choice of the constant $p \geq 2$ would be discussed later. In \eqref{e:heuristic_1}, the first term in the right-hand side comes from the dissipative effect of Laplacian and the other 
two terms come from the effects of the two terms in \eqref{e:main_terms} respectively. 
To make the last term small, one should choose $\lambda \geq \|w\|^{\frac{1}{1-\kappa}}$, but then the second term would be too large, no matter how we choose $p$. Actually, one expects that
\begin{align*}
	\|\lL_{\lambda} X\|_{L^p} \sim 
	\begin{cases}
		\sqrt{\log \lambda}\;, \quad & p < \infty\\ 
		\log \lambda\;, \quad & p = \infty\;,
	\end{cases}
\end{align*}
since the first case can be calculated explicitly if $p$ is an even integer and the second case is well-studied in the literature (see \cite{DRZ17} for example). In both cases, if $\lambda \geq \|w\|^{\frac{1}{1-\kappa}}$, because of the second term in \eqref{e:heuristic_1}, one 
obtains bounds of the type 
\begin{equ}
\label{e:double_exp_0}
	\d_t \|\hat w\|^2 \leq C \|\hat w\|^2 (\log \|\hat w\|)^{1+\eps}\;,
\end{equ}
which is still not good enough to get global well-posedness since $\eps > 0$. To handle this problem, the idea in \cite{HR24} is roughly as follows. First by the divergence free condition of $\hat w$ and $X$ one gets 
\begin{equ}
	\bracket{\hat w, \P \div(\hat w\otimes_s \lL_{\lambda} X)} = \bracket{\hat w, \grads \lL_{\lambda} X \cdot \hat w} \;.
\end{equ}
Exploiting the probabilistic structure of $X$, one can then show a bound of the type
\begin{equ}
	\bracket{\phi, \big(\Laplace + \grads \lL_{\lambda} X\big)\phi} \lesssim \|\phi\|^2 \log \lambda \;.
\end{equ}
Since we have $\|\hat w \|\sim \lambda$, in this way, one can push the $\eps$ in \eqref{e:double_exp_0} down to $0$ and obtain a double exponential bound 
for the solution. 

It seems hard to improve on this argument directly, so we choose to handle the term $\P\div(\hat w \otimes_s \lL_{\lambda} X)$ in another way. First note that it is not necessary to obtain a uniform in time bound for $w$, since it only serves as an intermediate step to give a good bound for the solution $u$ of \eqref{e:SNS}. Therefore, it may be possible to find some other remainder function for which energy estimates work better. 

For example, we can perform the following trick with It\^o's formula. Let $\Tilde{w} := w + \lL_{\lambda} \Tilde{X}$, then $\Tilde{w}$ satisfies
\begin{equ}
\label{e:w_tilde_equ}
    \d_t \Tilde{w} = \Laplace \Tilde{w} - \P \div \big(\Tilde{w}^{\otimes 2} + 2 w\otimes_s(\hH_{\lambda}\Tilde{X} + Y)- (\lL_{\lambda}\Tilde{X})^{\otimes 2} + Y^{\otimes 2}\big) + \lL_{\lambda}\xi_1 + \xi_2\;,
\end{equ}
with initial data $\Tilde{w}[0] = \lL_{\lambda} \ur[0] + \us[0]\,$. Still, the pairing $\langle{w, \P\div(w\otimes_s \hH_{\lambda} \Tilde{X})}\rangle$ is ill-defined in the energy estimate, so we consider the energy estimate for the smoother part $\Bar w = \Tilde{w} - \phi_{\lambda}$. We then 
find that, since we add the low frequency part $\lL_{\lambda} \Tilde X$ back to $w$, the worst term $\P\div(\hat w \otimes_s \lL_{\lambda} X)$ in previous discussion wouldn't appear in the equation for $\bar w$! It is transferred into another noise term $\lL_{\lambda} \xi_1$ instead, which affects the energy estimate through an Itô--Stratonovich correction term of order $\lambda^2$. This eventually yields a bound of the type
\begin{equ}
\label{e:heuristic_2}
	\d_t \|\bar w\|^2 \lesssim -\|\grad \bar w\|^2 + \alpha_0^2\lambda^2 + \|\grad \bar w\|\frac{\|\bar w\|_{L^4} \|w\|_{L^4} \|\Tilde{X}\|_{\cC^{-\kappa}}}{\lambda^{1-\kappa}}\;.
\end{equ}
If we compare this with \eqref{e:heuristic_1}, one can find that the second term in \eqref{e:heuristic_1} is replaced by the term $\alpha_0^2 \lambda^2$ now, which is some improvement if $\lambda \leq \|w\|$, since in this case it can be absorbed by the dissipation term provided that $\alpha_0$ is also small enough. (This is the main reason that the smallness of $\alpha_0$ is needed in the main theorem.)

However, as mentioned previously, to make the last term small enough one has to choose $\lambda \geq \|w\|^{\frac{1}{1-\kappa}}$, so these two terms again conflict with each other, and we want to somehow push the parameter $\kappa$ to zero to make these two thresholds match. 

\subsection{Ansatz}
To achieve this goal, we introduce another frequency scale $K$. From now on we fix $\lambda = \|w[0]\|$. Let $\lambda_+ := \lambda \vee 1$ and 
\begin{equ}[e:defK]
	K := \lambda_+ \vee \frac{\|w\|_{L^{12}}^{100}}{\lambda_+^{99}}\;.
\end{equ}
Then, instead of subtracting $\phi_{\lambda}$ as the roughest part, we subtract 
\begin{equ}
\label{e:equ_wh}
    \wh := - 2(\d_t-\Laplace)^{-1}\big(\P\div(
    w \para \hH_{K} \Tilde{X})\big)\;.
\end{equ}
Then the remainder $\wl := \Tilde{w} - \wh$ solves the equation
\minilab{e:equ_wl}
\begin{equs}
    \d_t \wl &= \Laplace \wl - \P\div((\wl)^{\otimes2}) + \lL_{\lambda}\xi_1 + \xi_2
    \label{e:equ_wl_NS}\\
    &\quad - 2\P\div(\wl \otimes_s\wh)- (\d_t-\Laplace)\wh - 2\P\div(w\para \hH_{\lambda} \Tilde{X})\label{e:equ_wl_big}\\
    &\quad \underbrace{{}- \P\div\big((\wh)^{\otimes2} + 2w \rpareq \hH_{\lambda} \Tilde{X} + 2w\otimes_s Y + Y^{\otimes2} - (\lL_{\lambda} \Tilde{X})^{\otimes2}\big)}\;.\qquad\label{e:equ_wl_small} \\
    &\qquad \qquad \qquad \qquad \qquad \qquad \; =: \rR
\end{equs}
Here the first line can be seen as a stochastic Navier--Stokes equation with smooth noise, the second line is the main diverging part in the equation, and the third line is a small remainder, henceforth denoted by $\rR$. 
Plugging the definition of $\wh$ from \eqref{e:equ_wh} into \eqref{e:equ_wl_big} we obtain 
\begin{equ}
    \eqref{e:equ_wl_big} = - 2\P\div\bigl(\wl\otimes_s\wh\bigr) - 2 \P \div \bigl(w \para \pP_{\lambda, K} \Tilde{X}\bigr)\;. 
\end{equ}

\begin{rmk}
We summarise our decomposition of the solution $u\in \cC^{-\kappa}$ as follows. First, $u = \Tilde{X} + Y + w$, where $\Tilde{X}$ is the linear solution \eqref{e:X_tilde_equ} with a $\cC^{-\kappa}$ rough initial data $\ur$, $Y$ solves the linear inhomogeneous equation \eqref{e:Y_equ}, and the remainder $w$ is expected to be in $\cC^{1-\kappa}$. We further decompose $w$ as $w = -\lL_{\lambda}\Tilde{X} + \wh + \wl$, where $\wh$ is given by \eqref{e:equ_wh} and $\wl$ satisfies equation \eqref{e:equ_wl}. More importantly, $\wl$ is now expected to be in $H^1$ so we can try to use energy estimates for $\wl$.
The two main features of the new ansatz proposed in this article are as follows.
\begin{enumerate}
    \item The term $- \lL_{\lambda} \Tilde{X}$ in the decomposition of $w$, which gives rise to the term $\lL_{\lambda} \xi_1$ in \eqref{e:equ_wl_NS}. The irregularity of $\xi$ then appears as the Itô--Stratonovich correction term in the It\^o formula.
    \item The introduction of a new time dependent frequency scale $K$. This suggests that the $\wh$ defined in the current way represents the high frequency part of $w$ better than the one defined in \cite{HR24}. 
    Note that we have to choose $\lambda$ in a time-independent way, or we have to also include terms 
    involving its time derivative in the equation for the remainder. The second frequency scale $K$, however, can be time dependent. 
\end{enumerate}
\end{rmk}

Now we use Itô's formula to get
\minilab{e:energy_estimate}
\begin{equs}
    \frac{\md}{\md t}\Big(\frac{1}{2}\|\wl\|^2\Big)
    &= -\|\grad \wl\|^2 + \bracket{\wl, \circ\, (\lL_{\lambda} \xi_1 + \xi_2)} + \bracket{\wl, \rR}\label{e:energy_estimate_perturbation}\\
     &+  \bracket{\grad\wl,  2 w \para \pP_{\lambda, K} \Tilde{X}} + \bracket{\grad \wl, 2\wl\otimes_s \wh} \;. \label{e:energy_estimate_main}
\end{equs}
Here $\bracket{\wl, \circ\, (\lL_{\lambda} \xi_1 + \xi_2)}$ denotes Stratonovich integration. Since we have taken $K$ to be much larger than $\lambda$, $\wh$ is smaller so that the second term in \eqref{e:energy_estimate_main} is small enough, but the price is that we get an extra term $\bracket{\grad\wl,  2 w \para \pP_{\lambda, K} \Tilde{X}}$. For this term, one can check that $\|\pP_{\lambda, K} \Tilde X\|_{L^4} \sim \sqrt{\log \frac{K}{\lambda}}$, so that we roughly get
\begin{equ}
	|\bracket{\grad\wl,  2 w \para \pP_{\lambda, K} \Tilde{X}}| \lesssim \alpha_0 \|\grad \wl\| \|w\|_{L^4} \sqrt{\log \frac{\|w\|_{L^{12}}}{\|w\|} }\leq \delta \|\grad \wl\|^2 + C\alpha_0 \|w\|^2\;.
\end{equ}
Therefore, if $\alpha_0$ is small enough, these terms can be absorbed by the dissipation term and we get a uniform in time bound.

In practice, we would like to work before a certain stopping time for convenience. We then use the strong Markov property to 
iterate our bound and eventually get the sought after long time control. We define the stopping time
\begin{equ}
    \label{e:stopping_time}
    T := \inf\Big\{t\geq0: \|w[t]\| > 2 V_{\alpha_0}\big(u[0]\big) \vee 2 \text{ or } \|X\|_{\cC^{-\kappa}} > \alpha_0 \text{ or } \|Y\|_{L^{\infty}} > \alpha_0^2 \Big\}\;.
\end{equ}
Before time $T$, we then have 
\begin{equ}
	\|w\| \leq 2\lambda_+ + 2\alpha_0 \leq 4 \lambda_+ \;, \quad  \|X\|_{\cC^{-\kappa}} \leq \alpha_0 \;, \quad \|Y\|_{L^{\infty}} \leq \alpha_0^2 \;,\quad   \|\tilde X\|_{\cC^{-\kappa}} \leq c_0\alpha_0\;.
\end{equ}
Here for $\|w\|$ we use the fact that $\alpha_0 \leq 1$, and for $\|\Tilde X\|_{\cC^{-\kappa}}$, we use continuity of heat semigroup in $\cC^{-\kappa}$ and the fact that $\|\ur[0]\|_{\cC^{-\kappa}} \leq 2\alpha_0$. $c_0$ is some universal constant related to the heat semigroup. Our main bound is the following. 
\begin{prop}
\label{pr:main_bound}
	For any initial data $u[0] = \ur[0] + \us[0]$ with $\|\ur[0]\|_{\cC^{-\kappa}} \leq 2\alpha_0$ and $\|\us[0]\| = \lambda$\,, define $\wl$ and the stopping time $T$ as above. Then there exists some constant $C$ such that for any $0 < a, \alpha_0 < 1$, we have 
	\begin{equ}
	\label{e:main_bound}
		\EE \Big( e^{\frac{T\wedge a}{2}}\|\wl[T\wedge a]\|^2 - \|\wl[0]\|^2\Big) \leq C \alpha_0 \cdot a^{1-6\kappa} \lambda^2 + C a\;.
	\end{equ}
\end{prop}
A lower bound for $T$ independent of the size of $\lambda$ then follows which, when combined with 
the strong Markov property, yields global well-posedness. Moreover, with the lower bound for $T$, previous decompositions for $u$ and the following Cauchy--Schwarz inequality
\begin{equ}
	\left(\EE \|\wl[T\wedge a]\| \right)^2 \leq \EE \left( e^{\frac{T\wedge a}{2}}\|\wl[T\wedge a]\|^2 \right)\EE e^{-\frac{T\wedge a}{2}}\;,
\end{equ}
we obtain a bound similar to \eqref{e:main_result1}, but with the time $t$ replaced by a stopping time. 
A stopping time argument is then proved in Proposition~\ref{pr:exponential_decay_general} to upgrade this to a fixed deterministic time. We write the argument in a general way to make it possible to apply the idea to other equations. Finally we sharpen our method to obtain a uniform in time bound for $\EE V_{2\alpha_0}(u[t])^{2N}$ for arbitrarily large $N$. Keeping track of the dependence on $N$ gives us a stretched exponential tail bound for the invariant measure.

\section{Uniform in time bound}
\label{sec:bound}
\subsection{Stochastic objects}
We begin with some basic facts on the stochastic objects.
Recall the definition of the linearised solution $X$ in \eqref{e:linear_solution}.

\begin{lem}
\label{le:stochastic_object}
    There exists some constant $C$ such that for any $j, p \in \N^*$ and $t \geq 0$\,, we have  
    \begin{equ} 
    \label{e:stochastic_object}
        \Big(\EE \|\Delta_{j}X[t]\|_{L^{2p}}^{2p} \Big)^{\frac{1}{2p}}\leq C\sqrt{p \,(1 \wedge 2^{2j}t)}\cdot \alpha_0\;.
    \end{equ}
\end{lem}
\begin{proof}
	Note that $\Delta_{j}X$ can be represented by $\Delta_{j} X(x,t) =  \sum_{k} \rho_j(k) \Hat{X}(k,t) e_k(x)$, where 
    \begin{equ}
    	\Hat{X}(k,t) =  \int_0^t e^{-|k|^2(t-s)}\phi_k \md B_k(s)\;.
    \end{equ}
    This means that $\Delta_{j}X(x, t)$ is a Gaussian random variable with variance at most
    \begin{equ}
        \alpha_0^2 \sum_{k} \rho_j^2(k) \int_{0}^{t}e^{-2|k|^2(t-s)}\md s \leq 100 \alpha_0^2(1 \wedge 2^{2j} t) \;.
    \end{equ}
    Then by Fubini and Nelson's estimate we have
    \begin{equ}
        \EE \|\Delta_{j}X[t]\|_{L^{2p}}^{2p}  = \int_{\T^2} \EE |\Delta_{j}X(x,t)|^{2p} \md x \leq C^p p^p \alpha_0^{2p}(1 \wedge 2^{2j} t)^{p}\;,
    \end{equ}
    which concludes the proof of \eqref{e:stochastic_object}.
\end{proof}

\begin{lem}
\label{le:Tilde_X_bound}
	For any $2 \leq p \leq \frac{1}{\kappa}$\,, there exists some $C(p)$ such that for any $j \geq 0$ and $0 \leq a \leq 1$, we have 
	\begin{equ}
		\EE\int_0^a \|\Delta_j \Tilde{X}[t]\|_{L^p}^p \md t \leq C(p) \alpha_0^p\cdot a^{1-\kappa p}\;.
	\end{equ}
\end{lem}
\begin{proof}
	Note that $\Tilde{X}[t] = e^{t\Delta}(\hH_{\lambda} \ur) + X[t]$. For the first term, by Besov embedding and heat flow estimate \eqref{e:heat_flow2} we have 
\begin{equ}
	\|e^{s \Delta}(\hH_{\lambda} \ur)\|_{L^{\infty}} \lesssim \|e^{s \Delta}(\hH_{\lambda} \ur)\|_{\cC^{\kappa}} \lesssim \alpha_0 s^{-\kappa}\;.
\end{equ}
The lemma follows by combining this with Lemma~\ref{le:stochastic_object}.
\end{proof}
\begin{cor}
\label{cor:lambda_norm}
	For any $\kappa > 0$, $p\in [1, \infty)$ and $0 < a \leq 1$ we have 
	\begin{equ}
		\sup_{\lambda \geq 1} \lambda^{p \kappa} \EE \int_0^a \| \hH_{\lambda}\Tilde X\|_{\bB^{-\kappa}_{p, p}}^p \md t \leq C(p)\alpha_0^p a^{1 - \kappa p} \;.
	\end{equ}
\end{cor}
\begin{rmk}
	This means that after taking expectation, $\Tilde X$ behaves like a function with regularity $0$ instead of $-\kappa$. 
\end{rmk}

Although a noise $\xi$ satisfying Assumption~\ref{as:noise} is not space-time white noise, the property of renormalisation is the same as the case of space-time white noise. We omit the proof of the following
and refer the reader to \cite[Appendix]{DPD02}.
\begin{prop}
\label{pr:wick_renormalization}
    Let $X_{N} := \lL_{N} X$. Then there exist a stochastic process $\Wick{X^{\otimes 2}}$ such that for any $T>0, \kappa>0$,
    \begin{equ}
    \label{e:wick_product}
        \big(\Wick{X_N^{\otimes 2}}\big) := X_N^{\otimes 2} - \EE X_N^{\otimes 2} \rightarrow \Wick{X^{\otimes 2}} \quad\text{in}\quad C([0,T]; \cC^{-\kappa})\;.
    \end{equ}
    Since $X_N$ is stationary in space, we have in particular 
        $\div(X_N^{\otimes 2}) \rightarrow \div(\Wick{X^{\otimes 2}})$ in $C([0,T]; \cC^{-1-\kappa})$.
    Moreover, for any $0<\eta<\frac{1}{2}$, $p\geq 1$ and small $\kappa'>0$, we have 
    \begin{equs}
    \label{e:linear_solution_short_time}
        \EE \sup_{s \leq t}\|X[s]\|_{\cC^{-2\eta}}^p &\leq C(p,\eta, \kappa') \alpha_0^p t^{p(\eta-\kappa')}\;, \\
    \label{e:wick_product_short_time}
        \EE \sup_{s \leq t}\|\div(\Wick{X^{\otimes 2}[s]})\|_{\cC^{-1-2\eta}}^p &\leq C(p,\eta,\kappa')\alpha_0^{2p} t^{p(\eta-\kappa')}\;.
    \end{equs}
\end{prop}\subsection{Energy estimate}
\label{sec:energy_estimate}
This subsection is devoted to proving Proposition~\ref{pr:main_bound}. We start with an estimate on $\wh$, defined in \eqref{e:equ_wh}.
\begin{lem}
\label{le:wh_bound}
    Let $T$ be the stopping time \eqref{e:stopping_time} and let $K$ be given by \eqref{e:defK}. For any $\alpha \geq \frac{1}{10}$ and $\gamma \in \R$ we have
    \begin{equ}
    	\label{e:wh_besov}
	\sup_{t\leq T} \|\wh\|_{\Besov{\gamma}{3, \infty}} \lesssim \lambda_+^{1-\alpha} \sup_{t\leq T} \|\hH_K \Tilde X\|_{\Besov{\gamma - 1 + \alpha}{4, \infty}}\;. 
    \end{equ}
\end{lem}
\begin{rmk}
	Here the choice of $3, 4$ in the Besov indices are related to the $L^{12}$ norm in the definition of $K$.
\end{rmk}
\begin{proof}
	By maximal regularity \eqref{e:Schauder1}, 
	\begin{equ}
		\sup_{t\leq T} \|\wh\|_{\Besov{\gamma}{3, \infty}} \lesssim \sup_{t\leq T} \|\P \div (w\para \hH_K \Tilde X)\|_{\Besov{\gamma-2}{3, \infty}} \lesssim \sup_{t\leq T} \|w \para \hH_K \Tilde X\|_{\Besov{\gamma - 1}{3, \infty}}\;.
	\end{equ}
	By the paraproduct estimate \eqref{e:para1}, we see that for $t \le T$ we have 
	\begin{equs}
		\|w \para \hH_K \Tilde X\|_{\Besov{\gamma - 1}{3, \infty}} &\lesssim \|w\|_{L^{12}} \|\hH_K \Tilde X\|_{\Besov{\gamma - 1}{4, \infty}} \\
		&\leq \frac{\|w\|_{L^{12}}}{K^{\alpha}} \|\hH_K \Tilde X\|_{\Besov{\gamma - 1 + \alpha}{4, \infty}} \leq \lambda_+^{1-\alpha} \|\hH_K \Tilde X\|_{\Besov{\gamma - 1 + \alpha}{4, \infty}} \;,
	\end{equs}
	where we used $\alpha \geq \frac{1}{10}$. The bound \eqref{e:wh_besov} follows. 
    \end{proof}

\begin{cor}
\label{cor:wl_bound}
	Let $T$ be the stopping time \eqref{e:stopping_time}. There exists a constant $C$ such that for any $0\leq t \leq T$, we have 
	$\|\wl[t]\| \leq C\lambda_+$.
\end{cor}
\begin{proof}
	Since $\wl = w - \wh$, by triangle inequality and definition of $T$, for $t\leq T$ we get 
	\begin{equ}
		\|\wl[t]\| \leq \|w[t]\| + \|\wh[t]\| \leq C\lambda_+ + \|\wh[t]\|\;.
	\end{equ}
	Combining it with Besov embedding and Lemma~\ref{le:wh_bound} we get 
	\begin{equ}
		\sup_{t\leq T} \|\wh\| \leq \sup_{t\leq T} \|\wh\|_{\Besov{\kappa}{3, \infty}} \leq C\lambda_+^{2\kappa}\sup_{t\leq T} \|\hH_K\Tilde X\|_{\Besov{-\kappa}{4, \infty}} \leq C\lambda_+^{2\kappa}\sup_{t\leq T} \| \Tilde X\|_{\Besov{-\kappa}{4, \infty}} \;.
	\end{equ}
	Note that by definition of $T$ we have $\sup_{t\leq T}\|\Tilde X[t]\|_{\cC^{-\kappa}} \leq C$. The result follows.
\end{proof}

\begin{lem}
\label{le:high_freq_func}
	Let $p \in [1, \infty]$, $\alpha \in \R$ and $\eps > 0$. There exists some constant $C$ such that for any $M \geq 1$, we have
	\begin{equ}
	\label{e:high_freq_func}
		\| \hH_M f \|_{\Besov{\alpha - \eps}{p, 1} } \leq C M^{-\eps} \| f \|_{\Besov{\alpha}{p, \infty}}\;.
	\end{equ}
\end{lem}
\begin{rmk}
	This estimate improves \cite[Lemma~4.3]{HR24} since it enhances the integrability in the second index of the Besov norm.
\end{rmk}
\begin{proof}
	By the definition of Besov norm, the requirement for $f$ and H\"older's inequality, we have 
	\begin{equs}
		\|\hH_M f \|_{\Besov{\alpha - \eps}{p, 1} } &= \sum_{j \geq \log_2(M) - 1} 2^{(\alpha - \eps)j} \|\Delta_j f \|_{L^p} \\
		&\lesssim \| f \|_{\Besov{\alpha}{p, \infty}} \sum_{j \geq \log_2(M) - 1} 2^{- \eps j } 
		\lesssim M^{-\eps} \| f\|_{\Besov{\alpha}{p, \infty}} \;.
	\end{equs}
	The bound \eqref{e:high_freq_func} is proved.
\end{proof}

We now have the ingredients in place to prove Proposition~\ref{pr:main_bound}, which gives an upper bound for $\wl$.
\begin{proof}[Proof of Proposition~\ref{pr:main_bound}]
Note that 
\begin{equ}
	\frac{d}{dt}\left(e^{\frac{t}{2}} \|\wl\|^2\right) = e^{\frac{t}{2}} \Big(\frac{1}{2}\|\wl\|^2 + \frac{d}{dt}\|\wl\|^2\Big)\;.
\end{equ}
Combining this with \eqref{e:energy_estimate}, we get 
\minilab{e:energy_estimate2}
\begin{equs}
    e^{\frac{T\wedge a}{2}} &\|\wl[T\wedge a]\|^2 - \|w^{\lL}[0]\|^2\\ 
    &=\int_0^{T\wedge a} e^{\frac{t}{2}} \Bigg[ -\|\grad \wl\|^2 + \frac{1}{2} \|\wl\|^2+ \bracket{\wl, \circ\, (\lL_{\lambda} \xi_1 + \xi_2)} + \bracket{\wl, \rR}\qquad \\
     &\qquad \qquad + \bracket{\grad\wl,  2 w \para \pP_{\lambda, K} \Tilde{X}} + \bracket{\grad \wl, 2\wl\otimes_s \wh} \Bigg] \md t \;. 
\end{equs}
The Itô--Stratonovich correction term is bounded by $\alpha_0^2 \lambda^2 + C$ and, after taking
expectations, the martingale term will vanish. We also bound $e^{\frac{t}{2}}$ by a constant $e^{1/2}$ since $T\leq 1$. Therefore, it suffices to prove that, for some constant $C>0$,
\begin{equs}
\label{e:energy_estimate2_large_freq}
	\EE \int_0^{T\wedge a}\left(-\delta \|\grad \wl\|^2 + |\bracket{\grad \wl, 2\wl\otimes_s \wh}| \right) \md t &\leq C \alpha_0 \cdot a^{1-6\kappa} \lambda_+^2 + C a  \;, \quad \\
\label{e:energy_estimate2_middle_freq}
	\EE \int_0^{T\wedge a}\left(-\delta \|\grad \wl\|^2 + |\bracket{\grad \wl, w \para \pP_{\lambda, K} \Tilde X} | \right) \md t &\leq C \alpha_0 \cdot a^{1-6\kappa} \lambda_+^2 + C a \;, \qquad\\
	\EE \int_0^{T\wedge a}\left(-\delta \|\grad \wl\|^2 + |\bracket{\wl, \rR}|\right) \md t &\leq C \alpha_0 \cdot a^{1-6\kappa} \lambda_+^2 + C a \;. \qquad 
\end{equs}
For the last term, we actually prove the stronger bound
\begin{equ}
	\label{e:energy_estimate2_small}
	\sup_{0 \leq t \leq T\wedge a} \left(-\delta \|\grad \wl\|^2 + |\bracket{\wl, \rR}|\right)  \leq C \alpha_0 \lambda_+^2 + C\;.
\end{equ}
\step{Proof of (\ref{e:energy_estimate2_large_freq}):} By duality, H\"older inequality and interpolation, we get 
\begin{equs}
	\label{e:large_freq1}
	|\bracket{\grad \wl, \wl \otimes_s \wh}| &\leq \|\grad \wl\| \|\wl \otimes \wh\| \\
	&\lesssim \|\grad \wl\| \|\wl\|_{L^6} \|\wh\|_{L^3} \lesssim \|\grad \wl\|^{5/3} \|\wl\|^{1/3} \|\wh\|_{L^3} \;.
\end{equs}
Corollary~\ref{cor:wl_bound} and Young's inequality allow us to bound this by
\begin{equ}
	\delta\|\grad \wl\|^2 + C \lambda_+^2 \|\wh\|^6_{L^3}\;,
\end{equ}
so that it suffices to show that 
\begin{equ}
	\EE \int_0^{T\wedge a} \|\wh\|^6_{L^3} \md t \leq \EE \int_0^{T\wedge a} \|\wh\|^6_{\Besov{0}{3, 1}} \md t \leq C \alpha_0 \cdot a^{1-6\kappa} \;.
\end{equ}
Note that $\wh$ only has frequencies larger than $\lambda_+$, by Lemma~\ref{le:high_freq_func}, the maximal regularity estimate \eqref{e:Schauder2}, the paraproduct estimate \eqref{e:para1}, and 
the bound for Littlewood--Paley projections \eqref{e:LP_estimate}, we get that 
\begin{equs}
	\int_0^{T\wedge a}\|\wh\|_{\Besov{0}{3, 1}}^6 \md t &\lesssim \int_0^{T\wedge a} \lambda_+^{-3/2} \|\wh\|_{\Besov{1/4}{3, \infty}}^6 \md t \\
	&\lesssim \lambda_+^{-3/2} \int_0^{T\wedge a} \|\P \div(w \para \hH_K \Tilde X)\|_{\Besov{-7/4}{3, \infty}}^6 \md t  \\
	&\lesssim \lambda_+^{-3/2} \int_0^{T \wedge a} \|w\|_{L^{12}}^6 \|\hH_K \Tilde X\|_{\Besov{-3/4}{4, \infty}}^6 \md t \\
	\label{e:large_freq2}
	&\lesssim \lambda_+^{-3/2} \int_0^{T \wedge a} \|w\|_{L^{12}}^6 K^{-\frac{3}{2}} \|\hH_K \Tilde X\|_{\Besov{-1/2}{4, \infty}}^6 \md t\;.
\end{equs}
Recall that by definition, $\frac{K}{\lambda_+} \geq \left(\frac{\|w\|_{L^{12}}}{\lambda_+}\right)^{100}$. Using this, Besov embedding and Corollary~\ref{cor:lambda_norm} we get that 
\begin{equs}
	\EE \int_0^{T\wedge a} \|\wh\|^6_{\Besov{0}{3, 1}} \md t &\lesssim \lambda_+^3 \EE \int_0^{T\wedge a} \|\hH_K \Tilde X\|_{\Besov{-1/2}{4, 6}}^6 \md t \\
	\label{e:large_freq3}
	&\leq \lambda_+^3 \EE \int_0^{T\wedge a} \|\hH_K \Tilde X\|_{\Besov{-1/2}{6, 6}}^6 \md t \lesssim \alpha_0^6 a^{1-6\kappa}\;.
\end{equs}
The bound is proved, since $\alpha_0 \leq 1$.

\step{Proof of (\ref{e:energy_estimate2_middle_freq}):} Setting $\beta = (\frac{1}{2} - \frac{500\kappa}{3})^{-1}$, we claim that for any $0 \leq t \leq T\wedge a$ we can prove
\begin{equ}
\label{e:middle_freq1}
	|\bracket{\grad \wl, w\para \pP_{\lambda, K} \Tilde X}| \lesssim \delta \|\grad \wl\|_+^2 + \lambda_+^2 \left(\lambda_+^{\kappa} \|\hH_{\lambda} \Tilde X\|_{\Besov{-\kappa}{4, \infty}} \right)^\beta \;.
\end{equ}
Once \eqref{e:middle_freq1} is shown, it follows from Besov embedding that 
\begin{equ}
	\EE \int_0^{T\wedge a} \left(\lambda_+^{\kappa} \|\pP_{\lambda, K} \Tilde X\|_{\Besov{-\kappa}{4, \infty}} \right)^{\beta} 
	\lesssim \EE \int_0^a \lambda_+^{\kappa \beta} \|\hH_{\lambda} \Tilde X\|^{\beta}_{\Besov{-\kappa}{4, \beta}} \md t \;.
\end{equ}
The bound \eqref{e:energy_estimate2_middle_freq} then follows from Corollary~\ref{cor:lambda_norm}, combining with the fact that $a, \alpha_0 \leq 1$.

Now we focus on proving \eqref{e:middle_freq1}. When $\|w\|_{L^{12}}\leq 1$, the left hand side is zero since $\pP_{\lambda, K} = 0$. Therefore we only need to consider the case where $\|w\|_{L^{12}} > 1$, which means $K = \frac{\|w\|_{L^{12}}^{100}}{\lambda_+^{99}}$. 
By duality, paraproduct estimate \eqref{e:para1}, we get 
\begin{equ}
	|\bracket{\grad \wl, w\para \pP_{\lambda, K} \Tilde X}| \leq \|\grad \wl\| \|w \para \pP_{\lambda, K} \Tilde X\| 
	\leq \|\grad \wl\| \|w\|_{L^4} \|\pP_{\lambda, K} \Tilde X\|_{\Besov{0}{4, 2}} \;. 
\end{equ}
By Lemma~\ref{le:high_freq_func} and the bound for the Littlewood--Paley projections \eqref{e:LP_estimate}, we get
\begin{equ}
	 \|\pP_{\lambda, K} \Tilde X\|_{\Besov{0}{4, 2}} \lesssim \lambda_+^{-\kappa} \|\pP_{\lambda, K} \Tilde X\|_{\Besov{\kappa}{4, \infty}} \lesssim \lambda_+^{-\kappa} K^{2\kappa} \|\pP_{\lambda, K} \Tilde X\|_{\Besov{-\kappa}{4, \infty}} \;.
\end{equ}
By triangle inequality we get 
\begin{equs}
	\|\grad \wl\| \|w\|_{L^4} \lambda_+^{-\kappa} K^{2\kappa} &= \|\grad \wl\| \lambda_+^{-199\kappa} \|w\|^{200\kappa}_{L^{12}} \|w\|_{L^4} \\
	&\leq \|\grad \wl\| \lambda_+^{-199\kappa} \left(\|\wl\|_{L^{12}}^{200\kappa} + \|\wh\|_{L^{12}}^{200\kappa} + \|\lL_{\lambda} \Tilde X\|_{L^{12}}^{200\kappa}\right) \\
	& \qquad \qquad \qquad \qquad \left(\|\wl\|_{L^4} + \|\wh\|_{L^4} + \|\lL_{\lambda} \Tilde X\|_{L^4} \right)\;.
\end{equs}
By Besov embedding and Lemma~\ref{le:wh_bound}, before $T$ we have
\begin{equs}
	\|\wh\|_{L^{12}} &\lesssim \|\wh\|_{\Besov{\frac{1}{2} + \kappa}{3, \infty}} \lesssim \lambda_+^{\frac{1}{2} + 2\kappa}\;, \\
	\|\wh\|_{L^4} &\lesssim \|\wh\|_{\Besov{\frac{1}{6} + \kappa}{3, \infty}}\lesssim \lambda_+^{\frac{1}{6} + 2\kappa}\;.
\end{equs}
By interpolation and Corollary~\ref{cor:wl_bound}, before $T$ we get
\begin{equs}
	\|\wl\|_{L^{12}}^{200\kappa} + \|\wh\|_{L^{12}}^{200\kappa} + \|\lL_{\lambda} \Tilde X\|_{L^{12}}^{200\kappa} &\lesssim \|\grad \wl\|^{\frac{500 \kappa}{3}} \|\wl\|^{\frac{100\kappa}{3}} + (\lambda_+^{\frac{1}{2} + 2\kappa})^{200\kappa} + \lambda_+^{200\kappa^2}\\
	&\lesssim \|\grad \wl\|_+^{\frac{500 \kappa}{3}} \lambda_+^{\frac{100\kappa}{3}}\;, \\
	\|\wl\|_{L^4} + \|\wh\|_{L^4} + \|\lL_{\lambda} \Tilde X\|_{L^4} &\lesssim \|\grad \wl\|^{1/2} \|\wl \|^{1/2} + \lambda_+^{\frac{1}{6} + 2\kappa} + \lambda_+^{\kappa}\\
	&\lesssim \|\grad \wl \|_+^{1/2} \lambda_+^{1/2}\;.
\end{equs}
With this we deduce that 
\begin{equs}
	|\bracket{\grad \wl, w\para \pP_{\lambda, K} \Tilde X}| &\lesssim \|\grad \wl\|_+^{\frac{3}{2} + \frac{500\kappa}{3}} \lambda_+^{\frac{1}{2} - \frac{500\kappa}{3}} \lambda_+^{\kappa} \|\pP_{\lambda, K} \Tilde X\|_{\Besov{-\kappa}{4, \infty}} \\
	&\leq \delta \|\grad \wl\|_+^2 + \lambda_+^2 \left(\lambda_+^{\kappa} \|\pP_{\lambda, K} \Tilde X\|_{\Besov{-\kappa}{4, \infty}} \right)^{\beta} \\
	&\leq \delta \|\grad \wl\|_+^2 + \lambda_+^2 \left(\lambda_+^{\kappa} \|\hH_{\lambda} \Tilde X\|_{\Besov{-\kappa}{4, \infty}} \right)^{\beta}  \;,
\end{equs}
thus proving \eqref{e:middle_freq1}. 

\step{Proof of (\ref{e:energy_estimate2_small}):} Recall that 
        \begin{equ}
            \rR = \P\div\left((\wh)^{\otimes2} + 2w\rpareq \hH_{\lambda} \Tilde{X} + 2w\otimes_s Y + Y^{\otimes2} - (\lL_{\lambda} \Tilde{X})^{\otimes2}\right)\;.
        \end{equ}
Since $\wl$ is divergence free, for any function $f$ we can use integration by parts to get $\bracket{\wl, \P\div(f)} = -
\bracket{\grad \wl, f}$. Combining this with  Hölder's inequality, we find that $|\bracket{\wl, \rR}|$ is bounded by
\begin{equ}
\label{e:rR_bound_1}
    \|\grad \wl\|\cdot\left(\|\wh\|_{L^3}\|\wh\|_{L^6}+2\|w\rpareq \hH_{\lambda} \Tilde{X}\| + 2\|w\|\|Y\|_{L^{\infty}} + \|Y\|_{L^4}^2 + \|\lL_{\lambda} \Tilde{X}\|_{L^4}^2\right).
\end{equ}
We bound these terms one by one. 
\smallstep{Term \TitleEquation{\|\wh\|_{L^3}\|\wh\|_{L^6}}{L3L6}:} By Besov embedding and Lemma~\ref{le:wh_bound} we get that before $T$, we have 
\begin{equ}
	\|\grad \wl\| \|\wh\|_{L^3}\|\wh\|_{L^6} \leq \delta \|\grad \wl\|^2 + C\|\wh\|_{L^3}^2 \|\wh\|_{L^6}^2 \leq \delta \|\grad \wl\|^2 + C\alpha_0^4 \lambda_+^{\frac{2}{3} + 4\kappa}\;.
\end{equ}

\smallstep{Term \TitleEquation{2\|w\rpareq \hH_{\lambda} \Tilde{X}\|}{|w > H_\lambda X|}:} 
By paraproduct estimates \eqref{e:para2}, \eqref{e:reso} and triangle inequality, before $T$ we have 
\begin{equ}
    \|w\rpareq \hH_{\lambda} \Tilde{X}\| \lesssim \|w\|_{\Besov{2\kappa}{4,2}} \|\hH_{\lambda} \Tilde{X}\|_{\Besov{-\kappa}{4,\infty}} \lesssim \alpha_0\left(\|\wl\|_{\Besov{2\kappa}{4,2}} + \|\wh\|_{\Besov{2\kappa}{4,2}} + \|\lL_{\lambda}\Tilde{X}\|_{\Besov{2\kappa}{4,2}} \right) \;.
\end{equ}
Then, by Lemma~\ref{le:wh_bound} and Besov embedding, before $T$ we have 
\begin{equ}
	\|\wh\|_{\Besov{2\kappa}{4,2}} \lesssim \|\wh\|_{\Besov{3\kappa + \frac{1}{6}}{3, \infty}} \lesssim \lambda_+^{\frac{1}{6} + 4\kappa}\;.
\end{equ}
Combining these we get 
\begin{equs}
	\|w\rpareq \hH_{\lambda} \Tilde{X}\| \lesssim \alpha_0 \left(\|\grad \wl\|^{\frac{1}{2} + 2\kappa} \|\wl\|^{\frac{1}{2} - 2\kappa} + \lambda_+^{\frac{1}{6} + 4\kappa} + \lambda_+^{3\kappa} \right)\lesssim \alpha_0 \|\grad \wl\|_+^{\frac{1}{2} + 2\kappa} \lambda_+^{\frac{1}{2} - 2\kappa}\;.
\end{equs}
The bound then follows by applying Young's inequality. 

\smallstep{Term \TitleEquation{2\|w\|\|Y\|_{L^{\infty}}}{2|w||Y|}:} Before time $T$ we have 
\begin{equ}
	\|\grad \wl\|\|w\| \|Y\|_{L^{\infty}} \leq \delta \|\grad \wl\|^2 + C \lambda_+^2 \alpha_0^4 \;.
\end{equ}

\smallstep{Term \TitleEquation{\|Y\|_{L^4}^2 + \|\lL_{\lambda} \Tilde{X}\|_{L^4}^2}{..}:} Before time $T$ we have 
\begin{equ}
	\|\grad \wl\| \big(\|Y\|_{L^4}^2 + \|\lL_{\lambda} \Tilde{X}\|_{L^4}^2\big) \leq \|\grad \wl\| \big(\alpha_0^4 + \lambda_+^{2\kappa} \alpha_0^2 \big) \leq \delta \|\grad \wl\|^2 + C \alpha_0^4 \lambda_+^{4\kappa}\;.
\end{equ}
This gives the proof of \eqref{e:energy_estimate2_small}. Combining the discussions above, we complete the proof of Proposition~\ref{pr:main_bound}. 
\end{proof}

\subsection{Lower bound for stopping time}
In the rest of this section we show how to upgrade the main bound Proposition~\ref{pr:main_bound} to \eqref{e:main_result1}. First, we prove the following tail bound for $T^{-1}$. Recall the definition of $T$ in \eqref{e:stopping_time}. 
\begin{prop}
\label{pr:stopping_time_lower_bound}
    Suppose the initial data has a decomposition $u = \ur + \us$ with $\|\ur\|_{\cC^{-\kappa}} \leq 2\alpha_0$ and $\|\us\| = \lambda$\,, then for any $a > 0$ sufficiently small, we have 
    \begin{equ}
    \label{e:stopping_time_lower_bound}
        \PP(T^{-1} > a^{-1})\lesssim a^{\frac{1}{8}-\kappa}\;,
    \end{equ}
    where the proportional constant is independent of $\lambda$, $\alpha_0$ and $a$\,.
\end{prop}
\begin{proof}
Define the event 
\begin{equ}
	\fA(a) := \Bigl\{\sup_{t\leq a}\|Y[t]\|_{L^{\infty}}\leq \frac{1}{2}\alpha_0^2 \Bigr\}\cap \Bigl\{\sup_{t\leq a}\|X[t]\|_{\cC^{-\frac{1}{4}}} \leq \eps_0\alpha_0\Bigr\}\;,
\end{equ}
where $\eps_0 \leq 1$ is a small constant to be chosen later. Then by the definition of $T$, we have 
\begin{equ}
	\fA(a) \cap \{T < a\} \subset \{\|w[T\wedge a]\| \geq 2\lambda_+ \}\;,
\end{equ}
where we use the fact that $\|w[t]\|$ is continuous in $t$ up to the existing time, which follows from the standard fixed point argument. By Lemma~\ref{le:wh_bound}, if $\eps_0 \leq 1$ is small enough, on the event $\fA \cap \{T<a\}$ we have 
\begin{equs}
	\|\wl[T\wedge a]\|^2 - \|\wl[0]\|^2 &\geq \big(\|w[T\wedge a]\| - \|\wh[T\wedge a]\| - \|\lL_{\lambda} \Tilde X\|\big)^2 - \|\wl[0]\|^2 \\ 
	&\geq \big(2\lambda_+ - C\eps_0 \alpha_0 \lambda_+^{2\kappa} \big)^2 - \lambda^2 \geq \lambda_+^2 \;.
\end{equs}
Then using Proposition~\ref{pr:main_bound} and Markov inequality, we get 
\begin{equ}[e:boundT1]
	\PP \left(\fA(a) \cap \{T < a\}\right) \leq C a^{1- 6\kappa}\;.
\end{equ}
It remains to bound $\PP(\fA^c)$. Define event
\begin{equ}
    \fB(a) := \Bigl\{\sup_{t \leq a}\|X[t]\|_{\cC^{-\frac{1}{4}}}\leq \eps_0\alpha_0 \Bigr\} \cap \Bigl\{\sup_{t\leq a}\|\div\big(\Wick{X^{\otimes2}[t]}\big)\|_{\cC^{-\frac{5}{4}}} \leq \eps_0^2 \alpha_0^2 \Bigr\}\;.
\end{equ}
First we show that $\fB(a) \subset \fA(a)$, which means estimating $Y$. Using mild formulation we get that 
\begin{equ}
    Y[t] = \int_0^t e^{(t-s)\Delta}\left(\PP\div\big(2Y[s]\otimes_s \Tilde{X}[s] + \Wick{\Tilde{X}^{\otimes2}[s]}\big)\right)ds\;.
\end{equ}
By Schauder estimate we have
\begin{equs}
    \|Y[t]\|_{\cC^{\frac{1}{4}+\kappa}} 
    &\leq C\int_0^t (t-s)^{-\frac{3}{4}-\frac{\kappa}{2}}\|\div(Y[s]\otimes \Tilde{X}[s])\|_{\cC^{-\frac{5}{4}}}ds \\
    &+ C\int_0^t (t-s)^{-\frac{3}{4}-\frac{\kappa}{2}} \|\div(\Wick{\Tilde{X}^{\otimes 2}[s]})\|_{\cC^{-\frac{5}{4}}}ds\;.
\end{equs}
For the first line, by product estimate and that $\|X[s]\|_{\cC^{-\frac{1}{4}}}\leq 1$ on $\fB(a)$, we get that the integral is smaller than $Ct^{\frac{1}{4}-\frac{\kappa}{2}}\sup_{s\leq t}\|Y[s]\|_{\cC^{\frac{1}{4}+\kappa}}$\,.
For the second line, we first decompose 
\begin{equ}
    \div(\Wick{\Tilde{X}^{\otimes 2}[s]}) = \div(\Wick{X^{\otimes 2}[s]}) + \div\big(2X[s]\otimes_s e^{t\Delta}\ur[0] + (e^{t\Delta}\ur[0])^{\otimes2}\big) \;.
\end{equ}
On $\fB(a)$ we use the triangle inequality, so that
\begin{equ}
    \|\div(\Wick{\Tilde{X}^{\otimes 2}[s]})\|_{\cC^{-\frac{5}{4}}} \leq \eps_0^2\alpha_0^2+ 2\eps_0\alpha_0 \|e^{s\Delta}\ur[0]\|_{\cC^{\frac{1}{4}+\kappa}} + \|e^{s\Delta}\ur[0]\|_{L^{\infty}}^2\;.
\end{equ}
We then use the heat flow estimate \eqref{e:heat_flow2} and the fact that $\|\ur[0]\|_{\cC^{-\kappa}} \leq \alpha_0$ to bound the second line by
\begin{equ}
    C\alpha_0^2 \int_0^t (t-s)^{-\frac{3}{4}-\frac{\kappa}{2}}(1+s^{-\frac{1}{8}-\frac{\kappa}{2}} + s^{-\kappa})ds \leq C\alpha_0^2 t^{\frac{1}{8}-\kappa}\;.
\end{equ}
Taking $a$ small enough, we get that 
\begin{equ}
    \sup_{t\leq a}\|Y[t]\|_{\cC^{\frac{1}{4}+\kappa}} \leq \frac{1}{2} \sup_{t\leq a}\|Y[t]\|_{\cC^{\frac{1}{4}+\kappa}} + \frac{\alpha_0^2}{100}\;,
\end{equ}
which implies that $\fB(a) \subset \fA(a)$. Then by Proposition~\ref{pr:wick_renormalization}, we have 
\begin{equ}
	\PP(\fA(a)^c) \leq \PP(\fB(a)^c) \lesssim a^{\frac{1}{8} - \kappa}\;.
\end{equ}
Combining this with \eqref{e:boundT1}, the result is proved.
\end{proof}
As a result, there exists some $t_0 \leq 1$ independent of the size of initial data such that $\PP(T < t_0) < \frac{1}{10}$\,. We will fix such $t_0$ in the following. 
Now we prove a stopping time version of the main bound.
\begin{prop}
\label{pr:decay_1}
	If $\alpha_0$ is sufficiently small, then there exist constants $0 < \rho < 1$ and $C > 0$ such that for any $u[0]$ with $V_{\alpha_0}(u[0]) < \infty$, we have 
	\begin{equs}
    	\label{e:decay_1}
		\EE V_{\alpha_0}\big(u[T\wedge t_0]\big) &\leq \rho V_{2\alpha_0}(u[0]) + C\;, \\
	\label{e:short_time}
		\sup_{a \leq t_0} V_{2\alpha_0}\big(u[T\wedge a]\big) &\leq 2V_{\alpha_0}(u[0]) + C\;.
	\end{equs} 
\end{prop}
\begin{proof}
Take a decomposition of $\us[0] = \us + \ur$ with $\|\us\| = \lambda$ and $\|\ur\|_{\cC^{-\kappa}} \leq 2\alpha_0$. For \eqref{e:decay_1}, recall our ansatz $u = \hH_{\lambda} \big(X + e^{t\Delta}\ur\big) + Y + \wh + \wl$.
	By the triangle inequality, the heat flow estimate, and Lemma~\ref{le:wh_bound}, we get 
	\begin{equs}
		\|u - \hH_{\lambda} X\|
		&\leq  \|\hH_{\lambda} e^{t\Delta} \ur\| + \|Y\| + \|\wh\| + \|\wl\|\\
		\label{e:Lyapunov_bound}
		&\leq  \alpha_0 (T\wedge a)^{-\kappa/2} + C + C \lambda_+^{\kappa} + \|\wl\|\;.
	\end{equs}
    By the Cauchy--Schwarz inequality and Proposition~\ref{pr:main_bound}, we get
    \begin{equs}
    	\left(\EE \|\wl[T\wedge t_0]\| \right)^2 &\leq \EE \left( e^{\frac{T\wedge t_0}{2}}\|\wl[T\wedge t_0]\|^2 \right)\EE e^{-\frac{T\wedge t_0}{2}}\;\\
	& \leq \left(\big(1+C \alpha_0 \cdot t_0^{1-6\kappa}\big) \lambda^2 + C t_0\right)\EE e^{-\frac{T\wedge t_0}{2}}\;.
    \end{equs}
    By Proposition~\ref{pr:stopping_time_lower_bound} we find that $\EE e^{-\frac{T\wedge t_0}{2}} < 1$. Since 
    this bound is uniform over $\alpha_0 < 1$ and over $\lambda$, we can take $\alpha_0$ small enough so that 
    \begin{equ}
    	\left(\EE \|\wl[T\wedge t_0]\| \right)^2 \leq \rho \lambda^2 + C\;,
    \end{equ}
 for some $0 < \rho < 1$ and $C > 0$. Plugging this and Proposition~\ref{pr:stopping_time_lower_bound} into \eqref{e:Lyapunov_bound}, we get 
    \begin{equ}
    	\EE \|u[T\wedge t_0] - \hH_{\lambda} X[T\wedge t_0]\| \leq \rho \lambda^2 + C\;.
    \end{equ}
By the definition of $\hH_{\lambda}$ and $T$, we have $\|\hH_{\lambda} X[T\wedge t_0] \|_{\cC^{-\kappa}} \leq \alpha_0$, so that \eqref{e:decay_1} is proved.
    Regarding \eqref{e:short_time}, we decompose $u = X + e^{t\Delta} \ur + Y + w$ with $\|\ur\|_{\cC^{-\kappa}} \leq \alpha_0$. Then by definition of $T$, for any $a > 0$ we have
    \begin{equs}
    	\|X[T\wedge a] + e^{(T\wedge a)\Delta}\ur\|_{\cC^{-\kappa}} &\leq \|X\|_{\cC^{-\kappa}} + \|e^{t\Delta}\ur\|_{\cC^{-\kappa}} \leq 2\alpha_0\;, \\
	\|Y[T\wedge a] + w[T\wedge a] \| &\leq \alpha_0^2 + 2\lambda_+ \;,
    \end{equs}
and the bound~\eqref{e:short_time} follows.
\end{proof}

\subsection{Strong Markov property and iteration} 
\label{sec:Iteration}
In this subsection we aim to transfer the bound in Proposition~\ref{pr:decay_1} from the random time $T\wedge t_0$ to any deterministic time $t$, with some standard argument. We formalise it as the following abstract result.  
\begin{prop}
\label{pr:exponential_decay_general}
	Let $u\in C([0, \infty); \xX)$ be a strong Markov process with state space $\xX$ and natural filtration $\{\fF_t\}_{t\in [0, \infty)}$\,. Let $T: C([0, \infty); \xX) \rightarrow [0, \infty)$ be a stopping time satisfying the following properties. 
	\begin{enumerate}
	\item There exists some $t_0$ independent of $u[0]$ such that $\PP(T < t_0) \leq \frac{1}{10}\, ;$
	\item There exist some Lyapunov functions $V, \Tilde V: \xX \rightarrow [0, \infty)$ such that $\Tilde V \leq V$ and
	\begin{equs}
	\label{e:one_step_as}
		\EE V(u[T\wedge t_0]) &\leq \rho \Tilde V(u[0]) + C_1 \; , \\
	\label{e:short_time_as}
		\sup_{a \leq t_0} \EE \Tilde{V}(u[T\wedge a]) &\leq C_2 V(u[0]) + C_3\;,
	\end{equs}
	with constants $0 < \rho < 1$ and $C_1, C_2, C_3>0$.
	\end{enumerate}
	Then there exist constants $\gamma = \gamma(t_0, \rho, C_2), C= C(t_0, \rho) > 0$ such that for any $t \geq t_0$\,, we have
	\begin{equ}
	\label{e:exponential_decay_general}
		\EE \Tilde{V}(u[t]) \leq C C_2 e^{-\gamma t} \Tilde V(u[0]) + C \big(\log C_2 + 1\big) C_1 C_2 + 2C_3\;.
	\end{equ}
\end{prop}
\begin{rmk}
	From previous subsection we know that $(V, \Tilde{V}) = (V_{\alpha_0}, V_{2\alpha_0})$ satisfies the assumptions.
\end{rmk}
\begin{rmk}
A similar stopping time argument is used in \cite[Section 3.2]{KT22}, but they exploit the ``coming down from infinity'' property of $\Phi^4_2$ which is not possessed by Navier--Stokes 
equation. A further assumption \eqref{e:short_time_as} is proposed to handle this. However, to 
find a good $\Tilde{V}$ satisfying \eqref{e:short_time_as}, we actually rely on the fact that the Navier--Stokes equations do not require any counterterms at the level of the 
regularised equations. 
\end{rmk}
\begin{proof}
We fix the time $t > t_0$ throughout the proof. We only focus on the proof of \eqref{e:exponential_decay_general}, as the other one follows from exactly the same procedure.

\step{Step 1:} 
Define a series of stopping times as follows. Let $\taut{0} = 0$\,. When $\taut{i-1} \leq t - t_0$\,, we define 
\begin{equ}
	\label{e:tauone_i_defn}
	T^{(i)}:= T\big(u[\taut{i-1}+ \cdot]\big) \wedge t_0\;, \qquad \taut{i} := \taut{i-1} + T^{(i)}\;,
\end{equ}
otherwise we set $\taut{i} := \taut{i-1}$. It is straightforward to show that the $\taut{i}$ are indeed 
stopping times. Since $\{\taut{i}\}$ forms an increasing bounded sequence, the stopping time $\tau_t := \lim_{i\rightarrow \infty} \taut{i}$ exists. 
Define $N_t:= \inf\{i\in \N: \taut{i} = \taut{i+1} = \tau_t \}$\,. We claim that for any $t > 0$, we have
    \begin{equ}
        \label{e:tail_bound}
        \EE \exp( N_t) \leq e^{2\lceil \frac{t}{t_0} \rceil} \leq e^{\frac{2t}{t_0} + 2}\;. 
    \end{equ}
    For the proof, first note that random variables $N_t$ is increasing in $t$, so it suffices to analyze $N_{k t_0}$. For $k=2$, note that $\{N_{2t_0} > i\} = \{\taut{i} \leq t_0\}$. Then by strong Markov property we have
    \begin{equs}
        \PP(\taut{i}\leq t_0) &= \EE(\boldsymbol{1}_{\taut{i}\leq t_0}\cdot\boldsymbol{1}_{\taut{i-1}\leq t_0})\\ &=\EE\left(\EE\big(\boldsymbol{1}_{\taut{i}\leq t_0}\vert \fF_{\taut{i-1}}\big)\boldsymbol{1}_{\taut{i-1}\leq t_0}\right) \\
        \label{e:exponential_tail}
        &\leq \frac{1}{10}\PP(\taut{i-1}\leq t_0) \leq \frac{1}{10^i}\;.
    \end{equs}
    Thus, we have the tail bound $\PP(N_{2t_0}>i)\leq \frac{1}{10^i}$, which means that for any $t\leq 2t_0$\,, 
    \begin{equ}
         \EE \exp\bigl(N_{t}\bigr) \leq \EE \exp\left( N_{2t_0}\right) \leq e^2\;.
    \end{equ}
    Suppose the result holds for $t = (k-1)t_0$. Then, by the definition of $N_t$ and the strong Markov property, 
    we get that
    \begin{equ}
        \EE \left(\exp\big( N_{kt_0}-N_{(k-1)t_0} \big) \Big\vert \fF_{\tau_{(k-1)t_0}}\right) \leq \EE \left(\exp\big(N_{2t_0} \big) \right) \leq e^2\;,
    \end{equ}
    so that \eqref{e:tail_bound} follows from the induction hypothesis.

\step{Step 2:} 
Set $W(u[0]) = \EE V(u[T\wedge t_0])$. 
By the strong Markov property, \eqref{e:one_step_as}, and the fact that $\Tilde V \le V$, one then has on the event $\taut{i-1} \neq \tau_t $, the bound
\begin{equs}
    \EE\left(V\big(u[\taut{i}]\big) \Big\vert \fF_{\taut{i-1}}\right) &= W(u[\taut{i-1}])
    \leq \rho \Tilde V\big(u[\taut{i-1}]\big) + C_1 \\
    &\leq \rho V\big(u[\taut{i-1}]\big) + C_1\;.\label{e:one_step_3}
\end{equs}
Since we have the trivial bound $\taut{i} \leq i t_0$, \eqref{e:one_step_3} holds for any $i\leq \frac{t-t_0}{t_0}$. This means that for $i_t = \lfloor \frac{t-t_0}{t_0}\rfloor$, there exists some constant $0 < \gamma < 1$ depending on $\rho$ and $t_0$ such that
    \begin{equ}
    \label{e:exp_decay_1}
        \EE V\big(u[\taut{i_t}]\big)  \leq e^{-\gamma t} \Tilde V\big(u[0]\big) + \frac{C_1}{1-\rho}\;.
    \end{equ}
    As for $i>i_t$\,, note that the event $\{\taut{i-1}\neq \tau_t \}$ is $\fF_{\taut{i-1}}$-measurable. 
    Therefore, by the strong Markov property and \eqref{e:one_step_3}, we have
    \begin{equ}
    	\EE \left(V\big(u[\taut{i}]\big) ;\taut{i-1}\neq \tau_t \right) \leq \EE \left(\rho V\big(u[\taut{i-1}]\big)  + C_1;\taut{i-1}\neq \tau_t \right)\;.
    \end{equ}
    This implies that
    \begin{equs}
        \EE V\big(u[\taut{i}]\big)  &= \EE \left(V\big(u[\taut{i}]\big) ;\taut{i-1}\neq \tau_t \right) + \EE \left(V\big(u[\taut{i-1}]\big) ;\taut{i-1} = \tau_t \right)\\
        &\leq \rho \EE \left( V\big(u[\taut{i-1}]\big);  \taut{i-1}\neq \tau_t \right)  + C_1 \EE \left( \taut{i-1}\neq \tau_t \right) \\ & \qquad \qquad \qquad \qquad \qquad \qquad + \EE \left(V\big(u[\taut{i-1}]\big) ;\taut{i-1} = \tau_t \right) \\
        &\leq \EE V\big(u[\taut{i-1}]\big)  + C_1\PP\left(\taut{i-1}\neq \tau_t\right)\;.
    \end{equs}
    We can iterate this bound so that, for any $i\geq i_t$, we have
    \begin{equs}
        \EE V\big(u[\taut{i}]\big) &\leq \EE V\big(u[\taut{i_t}]\big) + C_1\sum_{i \geq i_t} \PP(\taut{i}\neq \tau_t) \\
        &= \EE V\big(u[\taut{i_t}]\big) + C_1\sum_{i\geq i_t} \PP(N_t > i)\\
        &\leq \EE V\big(u[\taut{i_t}]\big) + C_1 \EE N_t\;.
    \end{equs}
    From \eqref{e:tail_bound} and Jensen's inequality, we get $\EE N_t \leq \frac{2t}{t_0} + 2$.
    Combining this with \eqref{e:exp_decay_1} we get
    \begin{equ}
    \label{e:exp_decay}
        \EE V\big(u[\tau_t]\big) \leq e^{-\gamma t} \Tilde V\big(u[0]\big)  + \Big(\frac{2t}{t_0} + 2 + \frac{1}{1-\rho}\Big)C_1\;.
    \end{equ}
    
\step{Step 3:} We restart the process at $t - t_0 < \tau_t \leq t$. Define stopping times $\Tilde{\tau}^{(i)}_{t}$ as follows. We set $\Tilde{\tau}^{(0)}_{t} = \tau_t$ and, for $i>0$, we define
\begin{equ}
	\Tilde{\tau}^{(i)}_{t} := \big(\Tilde{\tau}^{(i-1)}_{t} + T\big(u[\Tilde{\tau}^{(i-1)}_{t} + \cdot]\big)\big)\wedge t\;.
\end{equ}
Let $\Tilde{N}_t := \inf \{i\in \N: \taut{i} \geq t\}$. The same argument as in step~2 yields
\begin{equ}
\label{e:tau_tilde_tail}
	\PP\left(\Tilde{\tau}^{(i)}_t < t \right) \leq \frac{1}{10^i}\;.
\end{equ}
Note that by \eqref{e:short_time_as} and strong Markov property we have
\begin{equs}
	\sum_{i\geq 1} \EE\left( \Tilde{V} \big(u[t]\big)\;; \Tilde{N}_t = i\right) & \leq \sum_{i\geq 1} \EE\left( \Tilde{V} \big(u[t]\big)\;; \Tilde{\tau}^{(i-1)}_t < t \right)\\
	\label{e:short_time1}
	& \leq \sum_{i\geq 1} \EE \left(C_2 V\big(u[\Tilde{\tau}^{(i-1)}_{t}]\big) + C_3 \;;\Tilde{\tau}^{(i-1)}_t < t \right)\;.
\end{equs}
By the strong Markov property, the assumption \eqref{e:one_step_as}, and the fact that $V \geq 0$, we get 
\begin{equs}
	\EE \left( V\big(u[\Tilde{\tau}^{(i-1)}_{t}]\big) \;; \Tilde{\tau}^{(i-1)}_t < t\right) &= \EE \left(\EE \Big( V\big( u[\Tilde{\tau}^{(i-1)}_t]\big)\boldsymbol{1}_{\Tilde{\tau}^{(i-1)}_t < t} \Big \vert \fF_{\Tilde{\tau}^{(i-2)}_t}\Big) \right) \\
	&\leq \EE \left(\EE \Big( V\big( u[\Tilde{\tau}^{(i-1)}_t]\big)\boldsymbol{1}_{\Tilde{\tau}^{(i-2)}_t < t} \Big \vert \fF_{\Tilde{\tau}^{(i-2)}_t}\Big) \right) \\
	&\leq \EE \left(\rho V\big(u[\Tilde{\tau}^{(i-2)}_t]\big)  + C_1;\Tilde{\tau}^{(i-2)}_t < t \right)\; .
\end{equs}
Therefore, iterating this, and using \eqref{e:tau_tilde_tail} we get that 
\begin{equs}
	\EE \left( V\big(u[\Tilde{\tau}^{(i-1)}_{t}]\big) \;; \Tilde{\tau}^{(i-1)}_t < t\right) &\leq \rho^{i-1} \EE V\big(u[\Tilde{\tau}^{(0)}_t]\big) + \sum_{\ell = 1}^{i-1} C_1\rho^{\ell-1}\PP\left(\Tilde{\tau}^{(i-1-\ell)}_t < t\right) \\
	&\leq \rho^{i-1}\EE V\big(u[\Tilde{\tau}^{(0)}_t]\big) + C_1 \left(\frac{\rho^{(i-1)/2}}{1-\rho} + \frac{(\frac{1}{10})^{(i-1)/2}}{\frac{9}{10}}\right)\;.
\end{equs}
Plugging this into \eqref{e:short_time1}, it follows that
\begin{equs}
	\sum_{i\geq 1} &\,\EE\left( \Tilde{V} \big(u[t]\big)\;; \Tilde{N}_t = i\right) \\
	&\leq  C_2
	\sum_{i\geq 1}\left(\rho^{i-1}\EE V\big(u[\Tilde{\tau}^{(0)}_t]\big) + C_1 \left(\frac{\rho^{(i-1)/2}}{1-\rho} + \frac{(\frac{1}{10})^{(i-1)/2}}{\frac{9}{10}}\right) \right) + C_3 \sum_{i \geq 1}\frac{1}{10^{i-1}}\\
	&\leq  \frac{C_2}{1-\rho}\EE V\big(u[\tau_t]\big) + \frac{10C_1C_2}{(1-\rho)(1-\sqrt{\rho})} + 2 C_3\;. 
\end{equs}
Combining this with \eqref{e:exp_decay}, we get that for any $t > 0$, we have
\begin{equs}
	\EE \Tilde V(u[t]) &\leq \frac{C_2}{1-\rho} \left(e^{-\gamma t} \Tilde V\big(u[0]\big)  + \Big(\frac{2t}{t_0} + 2 + \frac{1}{1-\rho}\Big)C_1\right) + \frac{10C_1C_2}{(1-\rho)(1-\sqrt{\rho})} + 2 C_3 \\
	&\leq \frac{C_2}{1-\rho} e^{-\gamma t} \Tilde V\big(u[0]\big) + C(t_0, \rho) \cdot t C_1C_2 + 2 C_3\;.
\label{e:exp_decay_general_0}
\end{equs} 

\step{Step 4:}
Now take $T_* = \frac{1}{\gamma} \log \frac{C_2}{1-\rho} + \frac{1}{\gamma}$. Since $\gamma$ only depends on $t_0, \rho$, for any $t \in [T_*, 2T_*]$ we get
\begin{equ}
	\EE \Tilde V(u[t]) \leq e^{-1} \Tilde V(u[0]) + C(t_0, \rho) \big(\log C_2 + 1\big) C_1 C_2 + 2C_3\;. 
\end{equ} 
Iterating this, we get that for any $t \geq T_*$, we have
\begin{equ}
	\EE \Tilde V(u[t]) \leq e^{- \lfloor \frac{t}{T_*} \rfloor} \Tilde V(u[0]) + C(t_0, \rho) \big(\log C_2 + 1\big) C_1 C_2 + 2C_3\;.
\end{equ}
The result follows by combining this with \eqref{e:exp_decay_general_0}.
\end{proof}
\begin{rmk}
	In Step 3 we rely crucially on the fact that $\rho < 1$. When $\rho \geq 1$, a similar result can be obtained (with a negative $\gamma$ in the statement), provided that the assumption \eqref{e:short_time_as} is strengthened to a pathwise inequality 
	\begin{equ}
		\sup_{a \leq t_0} \Tilde{V}(u[T\wedge a]) \leq C_2 V(u[0]) + C_3\;.
	\end{equ}
	This allows us to prove the exponential growth in time bound mentioned in Remark~\ref{rmk:large_alpha}, since the choice of $(V, \Tilde V) = (V_{\alpha_0}, V_{2\alpha_0})$ does satisfy this stronger assumption.
\end{rmk}

\section{Moment bounds}
\label{sec:tail_bound}
\subsection{Energy estimate}
Now we start to prove moment bounds on the solution. First we prove an energy estimate for the moments by modifying the argument in the previous section. 
\begin{prop}
	For any initial data $u[0] = \ur[0] + \us[0]$ with $\|\us[0]\| = \lambda$ and $\|\ur[0]\|_{\cC^{-\kappa}} \leq 2\alpha_0$, define $\wl$ and stopping time $T$ as previously. Then there exists some constant $C$ which is independent of $N$, such that for any $0 < a < 1$ and $\alpha_0$ small enough (the smallness is independent of $N$), we have 
	\begin{equs}
	\label{e:moment_main_bound}
		\EE \Big(e^{\frac{T\wedge a}{2}} \|\wl[T\wedge a]\|^{2N} &- \|\wl[0]\|^{2N} \Big) \leq (C\alpha_0)^{2N} \lambda_+^{2N}a + (CN)^{N}a \quad \\
		&+ C^N \lambda_+^{2N} \int_0^a \EE \left(\lambda_+^{\kappa} \|\pP_{\lambda, K} \Tilde X\|_{\Besov{-\kappa}{4, \infty}} \right)^{2N/(\frac{1}{2} - \frac{500\kappa}{3})} \md t \\
		&+ C^N \lambda_+^{2N} \int_0^{a} \EE \big( \lambda_+^{3N} \|\hH_{\lambda} \Tilde X\|_{\Besov{-1/2}{4, 6N}}^{6N} \big)\md t \;.
	\end{equs}
\end{prop}
\begin{rmk}
\label{rmk:moment_difficulty}
	Recall that $\Tilde X[t] = X[t] + e^{t\Delta} \ur$. Both parts will cause problems when trying to bound the third line of the expression. For the first part, if we apply Lemma~\ref{le:stochastic_object} on $X$, we only obtain a bound of order $(CN\alpha_0)^{6N} \lambda_+^{2N}$, while we would like instead to obtain a bound of order $\rho^{6N} \lambda_+^{2N}$ for some $\rho < 1$, as in the first term of \eqref{e:main_result2}. This would still require us to take $\alpha_0$ to be small in a way that is dependent on $N$, which is not what we want. 
	For the second part, if the initial data is only in $\cC^{-\kappa}$, one can solely get a bound of order $(CN\alpha_0)^{6N} \lambda_+^{2N + 6\kappa N}$, but we want the order of $\lambda_+$ to be $\lambda_+^{2N}$. 
	We will discuss how to solve these problems in Section~\ref{sec:concentration}.
\end{rmk}
\begin{rmk}
	Here the bound is actually valid for all $\lambda \geq 0$, but we will only use it when $\lambda \geq N^2$ is large enough. In this regime the problem mentioned in the previous remark can be solved. When $\lambda < N^2$ we will bound the solution trivially by the definition \eqref{e:stopping_time} of the stopping time $T$.
\end{rmk}
\begin{proof}
By Itô's formula we have
\begin{equ}
	\frac{\md}{\md t}\Big(e^{\frac{t}{2}}\big(\|\wl\|^2\big)^N\Big) = e^{\frac{t}{2}} \left( \frac{1}{2}\|\wl\|^{2N} + N \|\wl\|^{2N-2} \circ \frac{\md}{\md t} \|\wl\|^2 \right)\;.
\end{equ}
Plugging \eqref{e:energy_estimate} into it, we get
\begin{equs}
    \,&e^{\frac{T \wedge a}{2}} \|\wl[T \wedge a]\|^{2N} - \|\wl[0]\|^{2N}\\ 
    &= \int_0^{T\wedge a} e^{t/2} \Bigg[ \frac{1}{2}\|\wl\|^{2N} - 2N \|\wl\|^{2N-2} \|\grad \wl\|^2 \quad \\
    \label{e:N_moment_main}
    & + 2N \|\wl\|^{2N-2} \left(\bracket{\grad \wl, 2w \para \pP_{\lambda, K} \Tilde X} + \bracket{\grad \wl, 2\wl \otimes_s \wh} \right) \\
    \label{e:N_moment_small}
    & + 2N \|\wl\|^{2N-2} \bracket{\wl, \rR} + 2N \bracket{\|\wl\|^{2N-2} \wl, \circ\, (\lL_{\lambda} \xi_1 + \xi_2)} \Bigg] \md t \;. \qquad 
\end{equs}
Therefore, it suffices to bound \eqref{e:N_moment_main} and \eqref{e:N_moment_small}. For \eqref{e:N_moment_main}, it suffices to prove
\begin{equs}
	\,&N \|\wl\|^{2N-2} |\bracket{\grad \wl, 2w  \para \pP_{\lambda, K} \Tilde X}| \\
	\label{e:N_moment_1}
	&\leq \delta N \|\grad \wl\|^2 \|\wl\|^{2N - 2} + C^N 
+ C^N \lambda_+^{2N} \left(\lambda_+^{\kappa} \|\pP_{\lambda, K} \Tilde X\|_{\Besov{-\kappa}{4, \infty}} \right)^{2N/(\frac{1}{2} - \frac{500\kappa}{3})}  \;,  \qquad  \qquad \\
	\,&\int_0^{T\wedge a}N \|\wl\|^{2N-2} |\bracket{\grad \wl, 2\wl \otimes_s \wh}| \md t \\ &\leq \int_0^{T\wedge a} \Big(\delta N \|\grad \wl\|^2 \|\wl\|^{2N - 2} + C^N 
\label{e:N_moment_2}
+C^N \lambda_+^{2N} \big( \lambda_+^{\frac{1}{2}} \|\hH_{\lambda} \Tilde X\|_{\Besov{-1/2}{4, 6N}}\big)^{6N}\Big) \md t \;.
\end{equs}
Note that for \eqref{e:N_moment_1} we can prove a bound that is pointwise in time, but in \eqref{e:N_moment_2} we have to take averages in time.
For \eqref{e:N_moment_small}, it suffices to prove
\begin{equs}
\,& \EE \int_0^{T\wedge a} \Big( N \|\wl\|^{2N-2} \bracket{\wl, \rR} + N \bracket{\|\wl\|^{2N-2} \wl, \circ\, (\lL_{\lambda} \xi_1 + \xi_2)}  \Big) \md t \\
\label{e:N_moment_3}
&\leq \EE \int_0^{T \wedge a} \delta N \|\grad \wl\|^2 \|\wl\|^{2N - 2} \md t + (C\alpha_0)^{2N} \lambda_+^{2N}a + (CN)^{N}a\;. \qquad 
\end{equs} 
\step{Proof of (\ref{e:N_moment_1}):} 
By \eqref{e:middle_freq1} the left hand side is bounded by 
\begin{equ}
	N \|\wl\|^{2N-2} \left(\delta \|\grad \wl\|_+^2 + \lambda_+^2 \left(\lambda_+^{\kappa} \|\pP_{\lambda, K} \Tilde X\|_{\Besov{-\kappa}{4, \infty}} \right)^{1/(\frac{1}{2} - \frac{500\kappa}{3})} \right)\;.
\end{equ}
For the first term, we bound $\|\grad \wl\|_+^2$ by $C\|\grad \wl\|^2 + C$, and use Young's inequality to get 
\begin{equ}
	CN\|\wl\|^{2N-2} \leq \delta (N-1) \|\wl\|^{2N} + \Tilde{C}^N \;.
\end{equ}
For the second term, we use Young's inequality to get that it is bounded by
\begin{equ}
	\delta(N-1)\|\wl\|^{2N} + C^N \lambda_+^{2N} \left(\lambda_+^{\kappa} \|\pP_{\lambda, K} \Tilde X\|_{\Besov{-\kappa}{4, \infty}} \right)^{2N/(\frac{1}{2} - \frac{500\kappa}{3})}\;. 
\end{equ}

\step{Proof of (\ref{e:N_moment_2}):}
By \eqref{e:large_freq1} we get that 
\begin{equ}
	N \|\wl\|^{2N-2} |\bracket{\grad \wl, \wl \otimes_s \wh}| \lesssim N \|\grad \wl\|^{\frac{5}{3}} \|\wl\|^{2N - \frac{5}{3}} \|\wh\|_{L^3} \;. 
\end{equ}
By Corollary~\ref{cor:wl_bound} and Young's inequality, this is bounded by 
\begin{equs}
	\,&\big(\|\grad \wl\|^{\frac{5}{3}} \|\wl\|^{2N - 2}\big) \cdot \big(\lambda_+^{\frac{1}{3}} \|\wh\|_{\Besov{0}{3, \infty}} \big) \\
	& \leq \delta \|\grad \wl\|^{\frac{10N}{6N - 1}} \|\wl\|^{(2N - 2)\frac{6N}{6N - 1}} + C^N \lambda_+^{2N} \|\wh\|_{L^3}^{6N} \;.
\end{equs}
The first term can be bounded by the dissipation term. Regarding the second term, we argue similarly to \eqref{e:large_freq2} and \eqref{e:large_freq3}. By Lemma~\ref{e:high_freq_func}, the maximal regularity estimate \eqref{e:Schauder2}, the paraproduct estimate \eqref{e:para1}, and our definition for $K$, we get
\begin{equs}
	\int_0^{T\wedge a} \|\wh\|_{\Besov{0}{3, \infty}}^{6N} &\leq C^N \int_0^{T\wedge a} \lambda_+^{-\frac{3N}{2}}\|\wh\|_{\Besov{1/4}{3, \infty}}^{6N} \\
	&\leq C^N \lambda_+^{-\frac{3N}{2}}\int_0^{T\wedge a} \|w \para \hH_K \Tilde{X}\|_{\Besov{-3/4}{3, 6N}}^{6N} \md t \\
	&\leq C^N \lambda_+^{-\frac{3N}{2}}\int_0^{T\wedge a} \|w\|_{L^{12}}^{6N} \|\hH_K \Tilde X\|_{\Besov{-3/4}{4, 6N}}^{6N} \md t \\
	&\leq C^N \lambda_+^{-\frac{3N}{2}}\int_0^{T \wedge a} \|w\|_{L^{12}}^{6N} K^{-\frac{3N}{2}} \|\hH_K \Tilde X\|_{\Besov{-1/2}{4, 6N}}^{6N} \md t \\
	\label{e:moment_large_freq}
	&\leq C^N \lambda_+^{3N} \int_0^{T \wedge a} \|\hH_{\lambda} \Tilde X\|_{\Besov{-1/2}{4, 6N}}^{6N} \md t \;.
\end{equs}

\step{Proof of (\ref{e:N_moment_3}):} 
By \eqref{e:energy_estimate2_small} we get that  before $T$ we have 
\begin{equ}
	N \|\wl\|^{2N - 2} \bracket{\wl, \rR} \leq N \|\wl\|^{2N - 2}\left(\delta \|\grad \wl\|^2 + C\alpha_0 \lambda_+^2 + C\right)\;.
\end{equ}
By Young's inequality, we have
\begin{equ}
	N\|\wl\|^{2N-2} \big(\alpha_0 \lambda^2 + C \big) \leq C^{N-1} \big(\alpha_0 \lambda^2 + C\big)^N + \delta(N-1) \|\wl\|^{2N}\;.
\end{equ}
In order to bound the expectation of the Stratonovich integral, we rewrite it as an Itô integral (whose expectation vanishes), 
plus the integral of the cross-variation of the integrand with $\lL_\lambda\xi_1 + \xi_2$. In order to compute the latter, we note that
\begin{equ}
	\d_t \big(N \|\wl\|^{2N-2} \wl \big) = N(N-1) \|\wl\|^{2N - 4} \d_t \|\wl\|^2 \cdot \wl + N \|\wl\|^{2N-2} \cdot \d_t \wl\;.
\end{equ}
Plugging \eqref{e:energy_estimate} into it, one can compute that the cross-variation is bounded by
\begin{equ}
	\label{e:moment_2}
	\EE \int_0^{T\wedge a} \Big(CN(N-1) \|\wl\|^{2N-2} + \|\wl\|^{2N-2} \big(\alpha_0^2 \lambda^2 + C \big)\Big)\md t\;.
\end{equ}
To bound \eqref{e:moment_2}, by Young's inequality, we have
\begin{equs}
	N(N-1)\|\wl\|^{2N-2} &\leq C^{N-1} (N-1)^{N} + \delta(N-1)\|\wl\|^{2N} \;, \\
	N\|\wl\|^{2N-2} \big(\alpha_0^2 \lambda^2 + C \big) &\leq C^{N-1} \big(\alpha_0^2 \lambda^2 + C\big)^N + \delta(N-1) \|\wl\|^{2N}\;.
\end{equs} 
The bound \eqref{e:N_moment_3} then follows.
\end{proof}

\subsection{Concentration effect}\label{sec:concentration}
Each Littlewood--Paley block of $\hH_{\lambda} X$ either vanishes or comprises 
at least of the order of $\lambda^2$ Gaussian degrees of freedom, which leads to good concentration effects for their $L^p$ norms when $\lambda$ is large. 
We exploit these to solve the first problem mentioned in Remark~\ref{rmk:moment_difficulty}. 
We start from the following tail estimate.
\begin{lem}
\label{le:Lp_concentration}
	Let $p\in \N^*$. There exists some constant $C = C(p)$ such that for any $j, N \geq 1$ and $t > 0$\,, we have
	\begin{equ}
	\label{e:Lp_concentration}
		\EE \|\Delta_j X[t]\|_{L^{2p}}^{2N} \leq \big (C\alpha_0)^{2N} \big(1\wedge 2^{2j} t\big)^{N} \Big(1 + N^{N} 2^{-jN/p}\Big)\;.
	\end{equ}
\end{lem}
\begin{rmk}	
As one can see in the proof, the factor $2^{-jN/p}$ reflects a concentration effect. In the discrete case and $p = \infty$, a very precise result of a similar flavor was proved in \cite{DRZ17}, but the scaling is a bit different from the $p < \infty$ case. 
\end{rmk}
\begin{proof}
	Write $X = (X_1, X_2)$. It suffices to prove the result for $X_1$. Let $C_{1, j}(t) = \EE |\Delta_j X_1(t, x)|^2$ which is independent of $x$. Note that by the arguments in Lemma~\ref{le:stochastic_object}, we have the bound 
	\begin{equ}[e:boundCj]
		C_{1, j}(t) \leq C \alpha_0^2 \big(1\wedge 2^{2j} t\big)\;.
	\end{equ}
	When $p = 1$, we have 
	\begin{equ}
		\|\Delta_j X_1[t]\|^2 - C_{1, j}(t) = \int_{\T^2} \big(\Delta_j X_1(t, x)\big)^2 - C_{1, j}(t) \,\md x = \int_{\T^2} \Wick{\big(\Delta_j X_1(t, x)\big)^2} \,\md x\;.
	\end{equ}
	This is just the zeroth Fourier mode of $\Wick{\big(\Delta_j X_1(t, x)\big)^2}$, which is
	\begin{equ}
		\sum_{k_1 + k_2 = 0} \rho_j(k_1) \rho_j(k_2) \Wick{\Hat{X}_1(k_1, t) \Hat{X}_1(k_2, t)} = \sum_{k} \rho^2_j(k) \Wick{|\Hat{X}_1(k, t)|^2}\;.
	\end{equ}
	Therefore, by Wick's formula we have 
	\begin{equs}
		\EE \Big(\|\Delta_j X_1[t]\|^2 -C_{1, j}(t)\Big)^{2} &\leq C \EE \Bigl(\sum_{k} \rho^2_j(k) \Wick{|\Hat{X}_1(k, t)|^2}\Bigr)^2 \\
		 &= C \sum_{k} \EE \Bigl( \rho^2_j(k) \Wick{|\Hat{X}_1(k, t)|^2}\Bigr)^2 
		 \leq C \alpha_0^4 2^{-2j} \big(1\wedge 2^{2j} t\big)^2\,,
	\end{equs}
	so that Nelson's estimate \slash hypercontractivity yields
	\begin{equs}
		\EE \Big(\|\Delta_j X_1[t]\|^2 -C_{1, j}(t)\Big)^{N} &\leq (CN)^{N} \left(\EE \Big(\|\Delta_j X_1[t]\|^2 -C_{1, j}(t)\Big)^{2}\right)^{N/2} \\
		\label{e:L2_concentration}
		&\leq \Big(C\alpha_0^2 N 2^{-j}\Big)^{N}\big(1\wedge 2^{2j} t\big)^{N}\;.
	\end{equs}
	Note that \eqref{e:L2_concentration} is actually stronger than \eqref{e:Lp_concentration}.
	When $p \geq 2$, first note that by Jensen's inequality it suffices to show that for any $N\in \N^*$, we have 
	\begin{equ}
		\EE \|\Delta_j X[t]\|_{L^{2p}}^{2pN} \leq \big (C\alpha_0)^{2pN} \big(1\wedge 2^{2j} t\big)^{pN} \Big(1 + N^{pN} 2^{-jN}\Big)\;.
	\end{equ}
	We have the chaos expansion
	\begin{equ}
		\|\Delta_j X_1[t]\|_{L^{2p}}^{2p} = \sum_{\ell = 0}^{p} \frac{(2p)!}{2^l l! (2p - 2l)!} C_{1, j}(t)^{p-\ell} \int_{\T^2} \Wick{\big(\Delta_j X_1(t, x)\big)^{2\ell}} \; \md x\;,
	\end{equ}
so that, using \eqref{e:boundCj} and the fact that $C$ is allowed to depend on $p$, it suffices to prove 
	\begin{equ}
	\label{e:Lp_Wick}
		\EE \Big( \int_{\T^2} \Wick{\big(\Delta_j X_1(t, x)\big)^{2\ell}} \; \md x \Big)^N \leq \big (C\alpha_0)^{2\ell N} \big(1\wedge 2^{2j} t\big)^{\ell N} \Big(1 + N^{\ell N} 2^{-jN}\Big)\;.
	\end{equ}
	As in the $p = 1$ case, performing a Fourier transform and applying Wick's formula, we get
	\begin{equs}
		\EE \left(\int_{\T^2} \Wick{\big(\Delta_j X_1(t, x)\big)^{2\ell}}\; \md x \right)^2 &= \EE \left( \sum_{k_1 + \cdots + k_{2\ell} = 0} \Wick{\prod_{i=1}^{2\ell} \rho_j(k_i) \Hat{X}_1(k_i, t)}\right)^2  \\
		&\leq C(\ell)\sum_{k_1 + \cdots + k_{2\ell} = 0} \prod_{i = 1}^{2\ell} \Big( \rho_j^2(k_i) \EE |\Hat X_1 (k_i, t)|^2 \Big) \\
		&\leq C(\ell)\alpha_0^{4\ell} 2^{-2j}\big(1\wedge 2^{2j} t\big)^{2\ell}\;.
	\end{equs}	
	The bound \eqref{e:Lp_Wick} then follows by using Nelson's estimate.
\end{proof}
\begin{rmk}
	From the proof we see that we are essentially estimating tail bounds for Wick powers of Gaussian fields. Therefore, it may be possible to use the techniques of generalised Fernique theorem obtained in \cite{FO10} to give an alternative proof. 
\end{rmk}
\begin{rmk}
	The tail bound \eqref{e:exp_tail} may be improved if \eqref{e:Lp_concentration} is proved for some fractional $p \in (1, 2)$, but in the current proof we rely on $p$ being an integer to get the required decay in $j$.
\end{rmk}
\begin{cor}
\label{cor:Lp_concentration}
	There exists some constant $C$ such that for any $t \in [0, 1]$, $N\in \N^*$ and $\gamma > 0$, we have 
	\begin{equ}
	\label{e:Lp_concentration2}
		\sup_{\lambda \geq N^2}\EE \lambda^{2 \gamma N} \|\hH_{\lambda} X[t]\|_{\Besov{-\gamma}{4, 2N}}^{2N} \leq \big (C\alpha_0)^{2N}\;.
	\end{equ}
\end{cor}
\begin{proof}
	We expand the Besov norm as
	\begin{equ}
		\|\hH_{\lambda} X[t]\|_{\Besov{-\gamma}{4, 2N}}^{2N} = \sum_{j \geq -1} 2^{-j\gamma \cdot 2N}\|\Delta_j \hH_{\lambda} X[t]\|_{L^4}^{2N} \;.
	\end{equ}
	The result then follows from Lemma~\ref{le:Lp_concentration}.
\end{proof}
To address the second problem mentioned in Remark~\ref{rmk:moment_difficulty}, we define a new Lyapunov function \eqref{e:N_Lyapunov} where more requirements are added to the rough part
\begin{equs}
\label{e:N_Lyapunov}
	V^{(N)}_{\alpha}(u) := \inf\Bigl\{&\|\ur\|_{\cC^{-\kappa}} + \|\us\|: u = \ur + \us, \; \|\ur\|_{\cC^{-\kappa}} \leq \alpha, \\
	& \|\ur\|_{\Besov{-\kappa/N}{2, 2}} \leq \alpha\,, \int_0^1 \|e^{t\Delta} \ur\|_{\Besov{0}{4, pN}}^{pN} \md t \leq \alpha^{pN}, p \in [1, 6]\Bigr\}\;.
\end{equs}
The new requirements added in the second line guarantee that $\ur$ is regular enough. It remains to show that the solution for \eqref{e:SNS} can be separated in such a way. Since $\ur$ is essentially the high frequency part of $X[T]$, where $T$ is the stopping time \eqref{e:stopping_time}, this amounts to getting some bound for $X[T]$. We begin with the following bound for the supremum in time, which follows from a standard Kolmogorov type argument. The desired bound for $X[T]$ will be obtained in Proposition~\ref{pr:N_stochastic_object} later.
\begin{lem}
\label{le:sup_Lp_concentration}
	Let $p\in \N^*$. There exists some constant $C = C(p)$ such that for any $j, N \geq 1$, we have
	\begin{equ}
	\label{e:sup_Lp_moment}
		\EE \sup_{0\leq t \leq 1} \|\Delta_j X[t]\|_{L^{2p}}^{2N} \leq  2^{4j}(C\alpha_0)^{2N} \Big(1 + N^{N} 2^{-jN/p}\Big)\;.
	\end{equ}
	As a result, there exists $j_*$ such that for $j \geq j_*$, we have 
	\begin{equ}
	\label{e:sup_Lp_concentration}
		\PP\left(\sup_{0\leq t \leq 1} \|\Delta_j X[t]\|_{L^{2p}} \geq j^{1/2}\alpha_0  \right) \leq C 2^{-6j}\;.
	\end{equ}
\end{lem}
\begin{proof}
	Note that in law for any $0 \leq t' < t \leq 1$ we have 
	\begin{equ}
	\label{e:stationarity}
		X[t] - e^{(t-t')\Delta} X[t'] \eqlaw X[t-t']\;.
	\end{equ}
	By heat flow estimate \eqref{e:heat_flow3} and Lemma~\ref{le:Lp_concentration}, we have 
	\begin{equs}
		\EE \| e^{(t-t')\Delta} \Delta_j X[t'] - \Delta_j X[t'] \|_{L^{2p}}^{2N}& \leq \big(C 2^{2j} (t-t')\big)^{2N} \EE \|\Delta_j X[t']\|_{L^{2p}}^{2N} \\
		& \leq (C\alpha_0)^{2N} \big(2^{2j} (t-t')\big)^{2N} \Big(1 + N^{N} 2^{-jN/p}\Big)\;.
	\end{equs}
	Then by H\"older inequality and Lemma~\ref{le:Lp_concentration}, for any $|t - t'| \leq 2^{-4j}$, we get that $\EE \| \Delta_j X[t] - \Delta_j X[t'] \|_{L^{2p}}^{2N}$ is bounded by 
	\begin{equs}
		\,&\EE \| \Delta_j X[t] - \Delta_j X[t'] \|_{L^{2p}}^{2N}\\
		&\leq 2^{2N} \Big(\EE \|\Delta_j X[t] - e^{(t-t')\Delta}\Delta_j X[t']\|_{L^{2p}}^{2N} + \EE \|e^{(t-t')\Delta}\Delta_j X[t'] - \Delta_j X[t']\|_{L^{2p}}^{2N} \Big)\\
		&\leq \big(C\alpha_0\big)^{2N} \Big(1 + N^{N} 2^{-jN/p}\Big) \Big( \big(2^{2j} (t-t')\big)^{N} + \big(2^{2j} (t-t')\big)^{2N} \Big)\\
		\label{e:concentration_translation}
		&\leq \big(C\alpha_0\big)^{2N} \Big(1 + N^{N} 2^{-jN/p}\Big) (t - t')^{N/2} \;.
	\end{equs}
	Following the same argument as in \cite[Theorem~2.1]{RY99} we get that for any $0 \leq t \leq 1- 2^{-4j}$, 
	\begin{equ}
		\EE \Big(\sup_{s\in [t, t + 2^{-4j}]}\| \Delta_j X[s] \|_{L^{2p}}^{2N}\Big) \leq \big(C\alpha_0\big)^{2N} \Big(1 + N^{N} 2^{-jN/p}\Big) (t - t')^{N/2} \;.
	\end{equ}
	Note that 
	\begin{equ}
		\EE \sup_{t\in [0, 1]} \left\|\Delta_j X[t]\right\|^{2N}_{L^{2p}} \leq \sum_{1 \leq \ell \leq 2^{4j}} \EE \sup_{t\in [(\ell-1) 2^{-4j}, \ell 2^{-4j}]} \|\Delta_j X[t]\|^{2N}_{L^{2p}}\;.
	\end{equ}  
	The bound \eqref{e:sup_Lp_moment} then follows.
	As for \eqref{e:sup_Lp_concentration}, by Chebyshev inequality and \eqref{e:sup_Lp_moment}, for any $M > 0$, $N \in \N^*$ we get 
	\begin{equ}
		\PP\left(\sup_{0\leq t \leq 1} \|\Delta_j X[t]\|_{L^{2p}} \geq M \right) \leq M^{-2N} 2^{4j}(C\alpha_0)^{2N} \Big(1 + N^{N} 2^{-jN/p}\Big)\;.
	\end{equ}
	Choosing some $N \in [M^2\alpha_0^{-2}, 2M^2\alpha_0^{-2}]$, we have 
	\begin{equ}
		\PP\left(\sup_{0\leq t \leq 1} \|\Delta_j X[t]\|_{L^{2p}} \geq M \right) \leq 2^{4j} \left(\frac{C\alpha_0}{M} \right)^{2M^2\alpha_0^{-2}} + 2^{4j} \left(\frac{C}{2^{j/p}}\right)^{2M^2\alpha_0^{-2}}\;.
	\end{equ}
	Now let $M = \alpha_0 j^{1/2}$. If $j$ is large enough we have $\frac{C\alpha_0}{M} \vee \frac{C}{2^{j/p}} \leq 2^{-5}$. The bound \eqref{e:sup_Lp_concentration} then follows.
\end{proof}

\begin{prop}
\label{pr:N_stochastic_object}
	For any $N\geq 2$, there exists a random variable $\Lambda$ such that for any $p \in [1, 6]$, we have
	\begin{equ}
	\label{e:high_freq_condition}
		\sup_{0\leq s \leq 1}\|\hH_{\Lambda} X[s]\|_{\Besov{-\kappa/N}{2, 2}} \leq \alpha_0\;, \quad \sup_{0\leq s \leq 1} \int_0^1 \|e^{t \Delta} \hH_{\Lambda} X[s]\|_{\Besov{0}{4, pN}}^{pN} \md t \leq \alpha_0^{pN}\;.
	\end{equ}
There furthermore exist constants $C$ and $C_*$ such that
\begin{equ}
	\label{e:low_freq_condition}
		\EE \sup_{0\leq s \leq 1} \|\pP_{C_*N \log N, \Lambda} X[s]\|^{2N} \leq (C\alpha_0^2)^{N} \;.
	\end{equ}
\end{prop}

\begin{proof}
Define the random variable
\begin{equ}
	\log_2(\Lambda) := \sup \Bigl\{j: \sup_{0\leq s \leq 1} \|\Delta_j X[s]\|_{L^{4}} \geq \alpha_0 j^{\frac{1}{2}}\Bigr\} \vee j_*\vee C_*N \log N\;,
\end{equ}
where $C_*$ is some large constant to be determined later and $j_*$ is the constant in Lemma~\ref{le:sup_Lp_concentration}. 
Then for $j_0 \geq j_* \vee C_*N \log N$, by \eqref{e:sup_Lp_concentration} we have
\begin{equ}
\label{e:Lambda_tail}
	\PP(\log_2(\Lambda) \geq j_0) \leq \sum_{j = j_0}^{\infty} \PP\Bigl(\sup_{0\leq s \leq 1}\|\Delta_j X[s]\|_{L^4} \geq \alpha_0 j^{\frac{1}{2} } \Bigr)
	\leq \sum_{j = j_0}^{\infty} C 2^{-6j} \leq C 2^{- 6j_0}\;.
\end{equ}
Since by definition of $\Lambda$ we have $\sup_{j \geq 1} j^{-\frac{1}{2}} \sup_{0\leq s \leq 1}\|\Delta_j \hH_{\Lambda} X[s]\|_{L^4} \leq \alpha_0$, we can pick $C_*$ large enough (independent of $\alpha_0, N$) such that
\begin{equ}
	\sup_{0\leq s \leq 1} \|\hH_{\Lambda} X[s]\|_{\Besov{-\kappa/N}{2, 2}} \leq \alpha_0 \sum_{j \geq j_*\vee C_* N \log N} j^{1/2} 2^{-\kappa j / N} \leq \alpha_0 \; .
\end{equ}
By heat flow estimate, we have 
\begin{equs}
	\sup_{0\leq s \leq 1} \int_0^1 \|e^{t\Delta} \hH_{\Lambda} X[s]\|_{\Besov{0}{4, pN}}^{pN} \md t &\leq \sum_{j \geq \Lambda} \int_0^1 e^{-c 2^{2j}t \cdot pN} \sup_{0\leq s \leq 1}\|\Delta_j \hH_{\Lambda} X[s]\|_{L^4}^{pN} \md t\\
	&\leq \alpha_0^{pN} \sum_{j \geq C_* N \log N} \frac{j^{pN/2}}{c 2^{2j} \cdot pN} \;.
\end{equs}
We can pick $C_*$ large enough (independent of $\alpha_0, N$), such that for any $1\leq p \leq 6$, we have
\begin{equ}
	\sum_{j \geq C_*N \log N} \frac{j^{pN/2}}{c 2^{2j} \cdot pN} \leq 1\;. 
\end{equ}
Therefore, the bound \eqref{e:high_freq_condition} holds and it remains to show \eqref{e:low_freq_condition}. We use Plancherel to expand the $2N$th power and use Hölder's inequality to get 
\begin{equs}
	\,& \EE \sup_{0\leq s \leq 1} \|\pP_{C_*N \log N, \Lambda} X[s]\|^{2N}  \\
	&= \EE \sup_{0\leq s \leq 1} \Bigl(\sum_{j} \|\Delta_j X[s]\|^2 \one_{C_* N\log N < j \leq \Lambda} \Bigr)^{N} \\
	&\leq \sum_{j_1, \cdots, j_{N} } \EE \prod_{k=1}^{N} \sup_{0\leq s \leq 1} \|\Delta_{j_k} X[s]\|^2 \one_{C_*N\log N < j_k \leq \Lambda} \\ 
	&\leq \sum_{j_1, \cdots, j_{N} } \prod_{k=1}^{N} \Bigl(\EE \sup_{0\leq s \leq 1} \|\Delta_{j_k} X[t]\|^{4N}\Bigr)^{\frac{1}{2N}} \PP\left(C_* N\log N < j_k \leq \Lambda \right)^{\frac{1}{2N}}\;.
\end{equs}
By Lemma~\ref{le:sup_Lp_concentration} and \eqref{e:Lambda_tail}, each term in the product is bounded by 
\begin{equ}
	C\alpha_0^2 2^{2j_k/N} \big(1 + N 2^{-j_k} \big) 
	2^{- 3j_k /N} \boldsymbol{1}_{j_k \geq C_* N \log N} 
	\leq C\alpha_0^2 2^{-j_k/N}\boldsymbol{1}_{j_k \geq C_* N \log N} \;.
\end{equ}
Therefore, if $C_*$ is large enough, we get  
\begin{equ}
	\EE \|\hH_{C_*N \log N} \lL_{\Lambda} X[T]\|^{2N}  \leq (C\alpha_0)^{2N}\big( \sum_{j} 2^{-j/N}\boldsymbol{1}_{j_k \geq C_* N \log N} \big)^{N} \leq (C\alpha_0)^{2N}\;,
\end{equ}
thus proving the claim.
\end{proof}

\subsection{Proof of the main result}

Now we are ready to prove the moment bounds for the solution. Recall our new Lyapunov function \eqref{e:N_Lyapunov}. We can prove the following generalisation of Proposition~\ref{pr:decay_1}.
\begin{prop}
\label{pr:moment_decay_1}
	Fix $t_0$ such that $\PP(T < t_0) < \frac{1}{10}$\,. Let $N\in \N^*$ and $V^{(N)}_{\alpha_0}(u[0]) < \infty $\,. 
	If $\alpha_0$ is sufficiently small (independent of $N$), then we have
	\begin{equs}
	\label{e:moment_decay_1}
		\EE \left(V^{(N)}_{\alpha_0}(u[T\wedge t_0])\right)^N &\leq \rho \big(V^{(N)}_{2\alpha_0}(u[0])\big)^{N} + (CN)^{2N}\;, \\
	\label{e:moment_short_time}
		\sup_{a\leq t_0} \EE \bigl(V_{2\alpha_0}^{(N)}(u[T\wedge a])\bigr)^N &\leq 6^N \big(V^{(N)}_{\alpha_0}(u[0])\big)^{N} +(CN)^{2N} \;,
	\end{equs}
	where $0 < \rho < 1$ and $C>0$ are some constants independent of $N$.
\end{prop}
\begin{proof}
	Take a decomposition of $u[0] = \us + \ur$ with $\|\us\| = \lambda$ and $\ur$ satisfying the conditions in \eqref{e:N_Lyapunov}. 
	By Proposition~\ref{pr:N_stochastic_object} and the definition of $T$, there exists some random variable $\Lambda$ such that $ \hH_{\Lambda} X[T \wedge t_0]$ satisfies the conditions in \eqref{e:N_Lyapunov}, and $\lL_{\Lambda} X[T\wedge t_0]$ satisfies 
	\begin{equ}
		\EE \|\pP_{C_* N \log N, \Lambda} X[T\wedge t_0]\|^N \leq (C\alpha_0)^N \;.
	\end{equ} 
	Recalling our ansatz $u = \hH_{\lambda} \big(X + e^{t\Delta}\ur\big) + Y + \wh + \wl$,  H\"older's inequality yields
	\begin{align*}
		\|u - \hH_{\Lambda} X\|^{N} 
		&\leq (1 + \delta)\|\wl\|^{N} \\
		&+ (CN)^N\Big(\|\pP_{\lambda, \Lambda} X\|^N + \|\hH_{\lambda} e^{t\Delta} \ur\|^{N} + \|Y\|^{N} + \|\wh\|^N \Big)\;.
	\end{align*}
	By \eqref{e:moment_main_bound}, Corollary~\ref{cor:Lp_concentration} and our requirement for $\ur$, for $t_0 < 1$ and $\lambda \geq C_*N^2$ we have 
	\begin{equ}
		\EE \Big(e^{\frac{T\wedge t_0}{2}} \|\wl[T\wedge t_0]\|^{2N} - \|\wl[0]\|^{2N} \Big) \leq (\alpha_0^2 + (C\alpha_0)^{2N}) \lambda_+^{2N} + (CN)^{N}\;.
	\end{equ}
	Using the Cauchy--Schwarz inequality and Proposition~\ref{pr:stopping_time_lower_bound}, if $\alpha_0$ is small enough we get
	\begin{equ}
		\EE \|\wl[T\wedge t_0]\|^N \leq \rho\big(\lambda^2 + 1\big)^N +(CN)^{N/2}
	\end{equ}
	with some $0 < \rho < 1$.
	By Proposition~\ref{pr:N_stochastic_object} and the fact that $\lambda \geq C_* N^2$ we have 
	\begin{equ}
		\EE \|\pP_{\lambda, \Lambda} X[T \wedge t_0]\|^N \leq (C\alpha_0)^N \;.	
	\end{equ}
	By Proposition~\ref{pr:stopping_time_lower_bound} and the heat flow estimate~\eqref{e:heat_flow1} we have 
	\begin{equ}
		\EE \|\hH_{\lambda} e^{(T\wedge t_0) \Delta}\ur\|^N \leq C^N \EE (T\wedge t_0)^{-\kappa}\|\ur\|^N_{\Besov{-\kappa/N}{2, 2}} \leq (C\alpha_0)^N \EE (T\wedge t_0)^{-\kappa} \leq (C\alpha_0)^N\;.
	\end{equ}
	By definition of $T$ we have $\EE \|Y[T\wedge t_0]\|^N \leq \alpha_0^{2N}$. By Lemma~\ref{le:wh_bound} we get that
	\begin{equ}
		\EE \|\wh[T\wedge t_0]\|^N \leq \lambda_+^{\kappa N}\EE \| \hH_K \Tilde X[T\wedge t_0]\|_{\cC^{-\kappa}}^N \leq (C\alpha_0)^N\lambda_+^{\kappa N}\;.
	\end{equ}
	Combining the estimates above, we find that if $\lambda \geq C_* N^2$, the bound \eqref{e:moment_decay_1} holds. If $\lambda < C_* N^2$, we still define the rough part to be $\hH_{\Lambda} X[T \wedge t_0]$, but use the ansatz $u = X + e^{t\Delta} \ur + Y + w$. Then by H\"older inequality, the remainder has the bound
	\begin{equs}
		\|u - \hH_{\Lambda} X[T \wedge t_0]\|^N \leq C^N \Big(\|\pP_{C_* N \log N, \Lambda} X[T\wedge& t_0]\|^N + \|\lL_{C_* N\log N} X[T\wedge t_0]\|^N\\
		&+ \|e^{t\Delta} \ur \|^N + \|Y\|^N + \|w\|^N \Big)\;. 
	\end{equs}
	For the first line, the expectation of the first term is bounded by Proposition~\ref{pr:N_stochastic_object}. For the second term, we first use the Plancherel identity and H\"older's inequality to get 
	\begin{equs}
		\|\lL_{C_* N\log N} X[T\wedge t_0]\|^N & \leq C^N \Bigl(\sum_{j \leq C_* N \log N} \|\Delta_j X[T \wedge t_0]\|^2 \Bigr)^{N/2} \\
		&\leq (C N \log N)^{N/2} \sum_{j \leq C_* N \log N} \|\Delta_j X[T \wedge t_0]\|^N\;.
	\end{equs} 
	By Lemma~\ref{le:sup_Lp_concentration}, we have 
\begin{align*}
\sum_{j \leq C_* N \log N} \EE \Big(\sup_{0\leq t \leq 1} \|\Delta_j X[t]\|_{L^2}\Big)^{N} &\leq \sum_{j \leq C_* N \log N} \Big(\EE \Big(\sup_{0\leq t \leq 1} \|\Delta_j X[t]\|_{L^2}\Big)^{100N}\Big)^{\frac{1}{100}}\\
	&\leq (C\alpha_0)^N \sum_{j \leq C_* N \log N} 2^{j/50}  \big(1 + (50N)^N 2^{-jN/2}\big) \\
	&\leq (C\alpha_0 N)^N\;.
\end{align*}
In conclusion, we get 
\begin{equ}
	\EE \|\lL_{C_* N\log N} X[T\wedge t_0]\|^N \leq (C\alpha_0 N^{3/2} \log N)^N\;.
\end{equ}
For the second line, the bound for $\|e^{t\Delta} \ur \|^N + \|Y\|^N$ is the same as previously. Since now $\lambda < C_* N^2$, we have 
\begin{equ}
	\EE \|w[T\wedge t_0]\|^N \leq (C\lambda_+)^{N} \leq (CN)^{2N}\;.
\end{equ}
Combining the estimates above, the bound \eqref{e:moment_decay_1} is proved.

As for \eqref{e:moment_short_time}, we use the ansatz $u = X + e^{t\Delta} \ur + Y + w$. Taking the same $\Lambda$ as previously, we define the rough part to be $\hH_{\Lambda} X[T\wedge a] + e^{(T\wedge a)\Delta} \ur$, then by the triangle inequality and the definitions of $T$ and $\Lambda$, it satisfies the requirement for the rough part in \eqref{e:N_Lyapunov} for any $a > 0$. As for the low frequency part, by H\"older's inequality we have the bound 
\begin{equs}
	\,&\EE \|\lL_{\Lambda}X[T\wedge a] + Y[T\wedge a] + w[T\wedge a]\|^N \\
	&\leq 3^{N}\EE \left(\|\lL_{\Lambda}X[T\wedge a]\|^N + \|Y[T\wedge a]\|^N + \|w[T\wedge a]\|^N \right)\;.
\end{equs}
By the same argument as previously, it is bounded by $(CN)^{2N} + (6\lambda)^N$.
The result follows.
\end{proof}

Now we are ready to prove Theorem~\ref{thm:main_result}.
\begin{proof}[Proof of Theorem~\ref{thm:main_result}]
The bound \eqref{e:main_result1} follows from a combination of Propositions~\ref{pr:decay_1} and~\ref{pr:exponential_decay_general}. As for \eqref{e:main_result2}, applying Proposition~\ref{pr:exponential_decay_general} with
	\begin{equ}
		(V, \Tilde{V}) = \big((V^{(N)}_{\alpha_0})^N, (V^{(N)}_{2\alpha_0})^N\big)\;
	\end{equ}
	and combining it with Proposition~\ref{pr:moment_decay_1}, we get that there exists some $\gamma_N > 0$, such that
	\begin{equ}
		\EE (V_{2\alpha_0}^{(N)})^N \big(u[t]\big) \leq C^N e^{-\gamma_N t} \big(V^{(N)}_{2\alpha_0}\big)^N \big(u[0]\big) + (CN)^{2N}\;.
	\end{equ}
	Note that, from the definition \eqref{e:N_Lyapunov} of $V^{(N)}_{2\alpha_0}$, for any $x \in \cC^{-\kappa} \cap L^2$ and $N \geq 2$ we have 
	\begin{equ}
		V_{2\alpha_0}(x) \leq V^{(N)}_{2\alpha_0}(x) \leq \|x\| \;.
	\end{equ}
	The bound \eqref{e:main_result2} then follows.
\end{proof}

\section{Exponential Mixing}
\label{sec:exponential_mixing}
The uniqueness of the invariant measure is usually implied by the strong Feller property of the process and some support theorem. We first study these two properties and then prove an exponential mixing result by combining all the ingredients together. 


\subsection{Strong Feller property}
For strong Feller property of singular SPDEs, a very general result is given in \cite{HM18}. At first it does not cover the Navier--Stokes equation due to the nonlocality of the Leray projection, but later, \cite{ZZ17} filled this gap following the same strategy. 
Since our noise is a bit different than the space-time white noise considered in \cite{ZZ17}, and we are in 2D which will simplify the proof a lot, we choose to prove the strong Feller property again for completeness. For a function space $X$, we use $X_{\div}$ to denote the subspace of $X$ consisting of divergence free functions.
\begin{prop}	
\label{pr:Strong_Feller}
	The solution $u$ of \eqref{e:SNS} is a Markov process in $\cC^{-\kappa}_{\div}$ satisfying the strong Feller property. Moreover, the transition probability is continuous in the total variation norm.
\end{prop}
The rest of this subsection is to prove this proposition.
First we set up the framework as in \cite{HM18}. The solution of the SPDE \eqref{e:SNS} can be viewed as a random dynamical process on Banach space $\Bar{U} = \cC^{-\kappa}_{\div} \cup \{\infty\}$, where the state $\{\infty\}$ means blow-up of the solution. By Markov property it suffices to consider time up to 1. Let $\mM := C([0,1]; \cC^{-\kappa}_{\div})\times C([0,1]; \cC^{-2\kappa-1}_{\div})$ be the space of models. A general element in $\mM$ will be denoted by $\bPi$. We reformulate the Da Prato--Debussche trick as following to describe the Markov process.
\begin{prop}
\label{pr:solution_map_continuity}
	Let $\Psi_1: \Bar{U} \times \mM \rightarrow C([0,1]\,;\, \Bar{U})$ be the operator that maps $(u_0, (\Xi_1, \Xi_2))$ to the solution of 
	\begin{equ}
		\d_t u = \Laplace u + \P \div(u^{\otimes 2}) + 2\P \div(\Xi_1\otimes_s u) + \Xi_2\;, \quad u[0] = u_0\;.
	\end{equ}
	Then $\Psi_1$ is continuous and Fréchet differentiable. As a result, $\Psi(u_0, \Xi_1, \Xi_2) := \Psi_1(u_0, \Xi_1, \Xi_2) + \Xi_1$ is also continuous and Fréchet differentiable.
\end{prop}
\begin{proof}
	It follows directly from the fixed point argument in \cite{DPD02}.
\end{proof}

\begin{proof}[Proof of Proposition~\ref{pr:Strong_Feller}]
The solution for \eqref{e:SNS} can then be represented as $\Psi\big(u_0, \bxi)$, where $$\bxi := \big(X, \P \,\div(\Wick{X^{\otimes2}})\big) \in \mM \;.$$
Also, we can denote the flow to be
\begin{equs}
	\bPhi: &\; [0, 1] \times [0, 1] \times \Bar{U} \times \mM \rightarrow \Bar{U}\;,\\
	\bPhi: &\; (s, t, u, \bPi)\mapsto \Phi_{s, t}(u, \bPi) := \Psi(u, \bPi\vert_{[s, t]})[t]\;.
\end{equs}
Here $\Psi(u, \bPi\vert_{[s, t]})[t]$ means the value at time $t$ of the SPDE starting at time $s$ with data $u$.
To prove Proposition~\ref{pr:solution_map_continuity}, it suffices to verify Assumptions 1-5 in \cite{HM18}.
Assumption 1 is then verified by Proposition~\ref{pr:solution_map_continuity}. Now let 
\begin{equ}
	r: (t, u_0, \bPi) \rightarrow \|\Psi(u_0, \bPi)\|_{\cC([0, t];\, \cC^{-\kappa})}\;,
\end{equ}
then $r$ satisfies the Assumption 2 of \cite{HM18}. The verification also follows from the standard fixed point argument of the local well-posedness theory.
Then we wish to perturb the noise $\xi$ by a deterministic function $h$ in the Cameron--Martin space. The Cameron--Martin space $\hH$ for $\xi$ is
\begin{equ}
	\hH := \Bigl\{h = \sum_{k \neq 0} \hat{h}(k, t)e_k : \sum_{k\neq 0}\int_{0}^{1} |\phi_k|^2|\hat{h}(k, t)|^2 \md t < \infty \Bigr\}\;.
\end{equ}
In particular, if $\xi$ satisfies Assumptions \eqref{as:noise} and \eqref{e:noise_irreducible}, then $\hH = L^2_{\div}([0, 1]\times \T^2)$. 
Define $E_s = L^p_{\div}([0, s]\times \T^2)$ and $E = E_1 \subset \hH$, with some $2<p<\infty$ sufficiently large. Also define the shift map $\tau: E\times \mM \rightarrow \mM$ to be 
\begin{equ}
\label{e:shift_map}
	\tau: \big(h, \big(\Xi_1, \Xi_2 \big)\big) \rightarrow \big(h + \Xi_1, \Xi_2 + 2 \P \div(\Bar{h} \otimes_s \Xi_1) + \P \div(\Bar{h}^{\otimes 2})\big)\;,
\end{equ}
where $\Bar{h} := (\d_t - \Laplace)^{-1}h$\,. 
Here we use the smaller space $E$ instead of $\hH$ in order to satisfy Assumption 11 in \cite{HM18}, which will be discussed later. For any $h\in E$ we have
\begin{equ}
	\tau(h, \bxi(\omega)) =  \big(h + X, \P\, \div(\Wick{(X+h)^{\otimes 2}}) \big) = \bxi(\omega + h) \quad a.s.
\end{equ}
Then we verify that this $\tau$ satisfies Assumptions 3 and 4 in \cite{HM18}. For Assumption 3, we still need to verify that if $h[r] = 0$ for any $r \in [s, t]$, then 
\begin{equ}
	\Phi_{s, t}(u, \tau(h, \bxi)) = \Phi_{s, t} (u, \bxi)\;,\; \forall u \in U \;.
\end{equ}
It is this property that requires the Assumption 8 for SPDE in \cite{HM18}, which is not fulfilled by Navier--Stokes because of the presence of the Leray projection. However, the Leray projection only causes non-locality in space, while here the requirement is the locality in time, so actually there are no conflicts. Note that $\Phi_{s, t}(u, \bxi)$ is the $\eps \rightarrow 0$ limit of $v_{\eps}$ at time $t$, which satisfies  
\begin{equ}
	\d_t v_{\eps} = \Laplace v_{\eps} + \P \div(v_{\eps}^{\otimes 2}) + \lL_{\eps^{-1}}\xi\;, \quad v_{\eps}[s] = u\;.
\end{equ}
With the same reasoning we have $\Phi_{s, t}(u, \tau(h, \bxi)) = \lim_{\eps \rightarrow 0} v_{\eps}^h[t]$, where $v_{\eps}^h$ is the solution of 
\begin{equ}
\label{e:tilted_equ}
	\d_t v_{\eps}^h = \Laplace v_{\eps}^h + \P \div((v_{\eps}^h)^{\otimes 2}) + \lL_{\eps^{-1}}\xi + h\;, \quad v_{\eps}^h[s] = u\;.
\end{equ}
Since $h[r] = 0$ for any $r \in [s, t]$, we have $v_{\eps}[t] = v_{\eps}^h[t]$. Let $\eps \rightarrow 0$ verifies Assumption~3 in \cite{HM18}. Note that here we take advantage of the fact that the renormalisation of $\bxi$ does not cause renormalisation at the level of the equation, which makes things much easier. 
Regarding Assumption 4, it directly follows from the Fréchet differentiability of $\Psi$ and the form of $\tau$ in \eqref{e:shift_map}. Finally for Assumption 5, denote $J_{s, t}$ for the Fréchet derivative of $\Phi_{s, t}$ in its first variable at the point $\Phi_{0, s}(u, \bPi)$, so $J_{s, t}$ is a bounded operator from $\cC^{-\kappa}$ to $\cC^{2\kappa}$. Choosing $p$ large enough such that $L^p \subset \cC^{-\kappa}$, the Assumption 11 in \cite{HM18} is satisfied and the rest of the proof is the same as Theorem~4.8 there.
\end{proof}

\subsection{Support theorem}
For the support theorem, it also can be easily proved by hand since actually the renormalisation does not change the equation. First we have the following support theorem for the lifted noise $\bxi$. 
\begin{prop}
\label{pr:noise_support}
	For $\bxi$ and $\mM$ as above, one has
	\begin{equ}
		\supp \big(\bxi\big) = \overline{\Bigl\{(f,\P \div(f^{\otimes 2})): f \text{ is smooth}, \div(f) = \int_{\T^2} f(x) \, \md x = 0\Bigr\}}	\;,
	\end{equ}
where the closure is taken in $\mM$.
\end{prop}
\begin{proof}
	Denote the right-hand side by $\Bar{A}$\,. First we show that $A \subset \supp\big(\bxi\big)$. Let $f_N := P_{\leq N}f$ be the space Fourier truncation for $f$ up to frequency $N$ and $g_N = (\d_t-\Laplace)f_N$\,. Then by Girsanov theorem, for every $|k|\leq N$, $B_k(t) - \int_0^t \frac{\hat{g}_N(k,t)}{\phi_k}dt$ has the same law as $B_k(t)$ under measure 
	\begin{equ}
		\md \PP^g_N := \exp \Bigg(\sum_{k\neq 0} \Big(\int_0^1 \hat{g}_N(k, t)\phi^{-1}_k \md B_k(t) - \frac{1}{2}\int_0^1 \hat{g}_N^2(k, t)\phi^{-2}_k \md t \Big) \Bigg) \md \PP\;.
	\end{equ}
	Here we used the assumption \eqref{e:noise_irreducible} so that $\phi_k \neq 0$. This means that $\xi_N - g_N$ under $\PP_N^g$ is the same as $\xi_N$ under $\PP$, yielding that $X_N + f_N$ under $\PP_N^g$ is the same as $X_N$ under $\PP$.
Thus,
	\begin{equs}
		\, &\PP\left(\|\big(X_N,\P \div(X_N^{\otimes2})\big)-\big(f,\P\div(f^{\otimes2})\big)\|_{\mM} < \delta\right)\\
		 &= \PP^g_N\left(\|\big(X_N+f,\P\div\big((X_N+f)^{\otimes2}\big)\big)-\big(f,\P\div(f^{\otimes2})\big)\|_{\mM}<\delta \right) \\
		 \label{e:tilted_prob}
		 &= \PP_N^g\left(\|\big(X_N, \P\div(X_N^{\otimes 2} + 2\P\div(X_N\otimes_s f))\big)\|_{\mM} < \delta\right).
	\end{equs}
	It follows from Proposition~\ref{pr:wick_renormalization} that for any small $\eps$, if $N$ is large enough we get 
	\begin{equ}
	\PP\left(\|\big(X_N, \P\div(X_N^{\otimes2})\big)-\big(X, \P\div(\Wick{X^{\otimes2}})\big)\|_{\mM} \geq \delta \right) < \eps\;,
		\end{equ}
so that 
	\begin{equs}
		\label{e:ball_prob_lower_bound_1}
		\,&\PP\left(\|\big(X,\P\div(\Wick{X^{\otimes2}})\big)-\big(f,\P\div(f^{\otimes2})\big)\|_{\mM} < 2\delta\right) \\		
		&\geq \PP\left(\|\big(X_N,\P\div(X_N^{\otimes2})\big)-\big(f,\P\div(f^{\otimes2})\big)\|_{\mM} < \delta\right) - \eps \\
		&=\EE\left(\frac{\md\PP_N^g}{\md \PP} \boldsymbol{1}_{\|(X_N, \P\div(X_N^{\otimes 2} + 2\P\div(X_N\otimes_s f)))\|_{\mM} < \delta}\right) - \eps\;.
	\end{equs}
	Since $g$ is smooth, we have
	\begin{equ}
		\frac{\md \PP_N^g}{\md \PP} \rightarrow \frac{\md \PP^g}{\md \PP}:= \exp\left(\sum_{k\neq 0}\int_0^1 \hat{g}(k, t)\phi^{-1}_k \md B_k(t) - \frac{1}{2}\int_0^1 \hat{g}^2(k, t)\phi^{-2}_k \md t \right)\;.
	\end{equ}
	Let $N\rightarrow \infty$ and $\eps \rightarrow 0$. The dominated convergence theorem and Proposition~\ref{pr:wick_renormalization} yield for any $\delta > 0$
	\begin{equs}
		\PP&\left(\|\big(X,\P\div(\Wick{X^{\otimes2}})\big)-\big(f,\P\div(f^{\otimes2})\big)\|_{\mM} < 2\delta\right) \\ 
		&\geq \EE\left( \frac{\md \PP^g}{\md \PP} \boldsymbol{1}_{\|(X, \P\div(\Wick{X^{\otimes 2}}) + 2\P\div(X\otimes_s f))\|_{\mM} < \delta} \right) > 0\;,
	\end{equs}
	so that $A \subset \supp \big(\bxi \big)$.
	The converse inclusion follows immediately from Proposition~\ref{pr:wick_renormalization}.
\end{proof}
The following control result from \cite{CF96} is also needed.
\begin{prop}
\label{pr:control}
	For any $u_0, u_1 \in \cC^{-\kappa}_{\div}$ and any small $\delta>0$, there exists some smooth function $f$, such that $u := \Psi(u_0, f, \P\div(f^{\otimes 2}))$ satisfies $\|u[1]-u_1\|_{\cC^{-\kappa}} < \delta$\,.
\end{prop}
\begin{proof}
	In \cite{CF96} the exact controllability for smooth $u_0$ and $u_1$ is obtained, from which we directly get the approximate controllability for rougher functions $u_0, u_1$ as stated in this result. Since for any $u_0$ and $u_1$ in $\cC^{-\kappa}_{\div}$, we can always choose divergence free smooth functions $\Tilde{u}_0$ and $\Tilde{u_1}$ that are arbitrarily close to $u_0$ and $u_1$ in the $\cC^{-\kappa}$ norm respectively. Then by the exact controllability result, there exists some smooth function $f$, such that $\Psi(\Tilde{u}_0, f, \P \div(f^{\otimes 2}))[1] = \Tilde{u}_1$. For any $\delta > 0$, by Proposition~\ref{pr:solution_map_continuity} we can always choose $\Tilde{u}$ and $u$ close enough such that 
	\begin{equ}
		\| \Psi \big(\Tilde{u}_0, f, \P \div(f^{\otimes 2})\big)[1] - \Psi \big(u_0, f, \P \div(f^{\otimes 2})\big)[1] \|_{\cC^{-\kappa}} < \delta \;.
	\end{equ}
	The result then follows by using the triangle inequality.
\end{proof}
Now we are ready to prove the following support theorem for SPDE \eqref{e:SNS}.

\begin{prop}
	For every $u_0, u_1\in \cC^{-\kappa}_{\div}$ and $\delta > 0$, we have 
	\begin{equ}
	\label{e:accessible}
		\PP\left(\|\Psi\big(u_0, X, \P \div(\Wick{X^{\otimes 2}})\big)[1]-u_1\| < 2\delta \right) > 0\;.
	\end{equ}
\end{prop}
\begin{proof}
	First, we find some smooth function $f$ satisfying Proposition~\ref{pr:control}. Then by Proposition~\ref{pr:solution_map_continuity}, there exists some $\delta' > 0$, such that for every $(\Xi_1,\Xi_2) \in B\big((f,\div(f^{\otimes 2})),\delta' \big)$, we have $\|\Psi(u_0, \Xi_1, \Xi_2)[1]-u_1\| < \delta$. By Proposition~\ref{pr:noise_support}, such a ball has positive probability if $(\Xi_1, \Xi_2)$ is distributed as $\big(X, \P \div(\Wick{X^{\otimes 2}})\big. )$, which gives the result. 
\end{proof}

\subsection{Conclusion}
Now we start to prove Theorem~\ref{thm:exp_mixing}, using \cite[Theorem~3.6]{Hai10}. An important ingredient is the following proposition.
\begin{prop}
\label{pr:small_set}
	Fix a large time $t_*$\,. For any $K >0$\,, there exists some $\delta$ such that for any $u_0, \Tilde{u}_0 \in \cC^{-\kappa}$ such that $V_{2\alpha_0}(u_0), V_{2\alpha_0}(\Tilde{u}_0) \leq K$, we have
	\begin{equ}
		\|\pP_{t_*}(u_0, \cdot) - \pP_{t_*}(\Tilde{u}_0, \cdot)\|_{TV} \leq 1-\delta\;.
	\end{equ}
\end{prop}
To prove this, first we need to derive some a priori bound for $\cC^{-1+\kappa}$ norm of the solution after a short time. Note that by Sobolev embedding, the Lyapunov function $V$ only controls the $\cC^{-1}$ norm, so this step explores the smoothing effect of the Navier--Stokes flow. To do this, it would be more convenient to use an ansatz similar to the one in \cite{HR24} so that we don't have any martingale term. We still decompose $u[0] = \ur[0] + \us[0]$ with $\|\ur[0]\|_{\cC^{-\kappa}} \leq 2\alpha_0$ and $\|\us[0]\| = \lambda$. Then we define $\Bar{X}$ solving
	\begin{equ}
		\d_t \Bar{X} = \Laplace \Bar{X} + \xi_1 + \xi_2\;, \quad \Bar{X}[0] = \ur[0]\;.
	\end{equ}
	Let $\Bar{Y}$ solve 
	\begin{equ}
		\d_t \Bar{Y} = \Laplace \Bar{Y} + \P \div \left(2\Bar{X} \otimes_s \Bar{Y} + \Bar{X}^{\otimes 2}\right)\;, \quad \Bar{Y}[0] = 0\;.
	\end{equ}
	Let $\Bar{w} = u - \Bar{X} - \Bar{Y}$ and $\Bar{T} = \inf\{t \geq 0: \|\Bar{w}[t]\| \geq 2 \lambda_{+} \} \wedge 1$. Then we define $\bwl = \Bar{w} - \bwh$, where 
	\begin{equ}
		\bwh := 2(\d_t - \Laplace)^{-1}\left(\P \div \big(\Bar{w} \para \hH_{\lambda_+^3} \Bar{X}\big)\right)\;.
	\end{equ}
	We have the following proposition for $\bwl$.
\begin{prop}
	For any initial data $u[0] = \ur + \us$ with $\|\ur\|_{\cC^{-\kappa}} \leq 2\alpha_0$ and $\|\us\| = \lambda$\,, there exists some $C(\lambda)$ such that we have 
	\begin{equ}
	\label{e:bwl_bound}
		\PP\left(\|\bwl[\Bar{T}]\|_{H^{\kappa}}^2 \leq C(\lambda) \right) \geq \frac{9}{10}\;.
	\end{equ}
\end{prop}
\begin{rmk}
	The proposition generalises \cite[Lemma~5.3]{HR24} a bit, since we are not assuming the existence of the $H^{\kappa}$ norm of the initial data. In other words, this proposition describes the smoothing effect of the Navier--Stokes flow.
\end{rmk}
\begin{proof}
	The equation for $\bwl$ is 
	\begin{equ}
		\d_t \bwl = \Laplace \bwl + \P \div\Big(\Bar{w}^{\otimes 2} + \big(2\lL_{\lambda_{+}^3} \Bar{X} + 2 \Bar{Y}\big) \otimes_s \Bar{w} + 2 \Bar{w} \rpareq \hH_{\lambda_{+}^3} \Bar{X} + \Bar{Y}^{\otimes 2}\Big)\;.
	\end{equ}	
	Define $\L^{\kappa}_{\Bar{T}} := 1 + \sup_{0\leq s \leq t}\left(\|X[s]\|_{\cC^{-\kappa}} + \|Y[s]\|_{\cC^{2\kappa}} \right)$.
	As in Lemma~\ref{le:wh_bound} or \cite[Lemma 4.6]{HR24}, we can get that before $\Bar{T}$,  
	\begin{equ}
	\label{e:bwh_bound}
		\|\bwh\|_{H^{\frac{1}{2}}} \lesssim \|\Bar{w} \para \hH_{\lambda_{+}^3} \Bar{X} \|_{H^{-\frac{1}{2}}} \lesssim \|\Bar{w}\| \|\hH_{\lambda_{+}^3} \Bar{X}\|_{\Besov{-\frac{1}{2}}{\infty, 2}} \lesssim \L^{\kappa}_{\Bar{T}}\;.
	\end{equ}
	Then it follows that 
	\begin{equs}
		\Big| \bracket{\P \div\big(\Bar{w}^{\otimes 2}\big), \bwl} \Big| &= \Big| 2\bracket{\bwl \otimes_s \bwh, \grad \bwl} + \bracket{\big(\Bar{w}^{\hH}\big)^{\otimes 2}, \grad \bwl} \Big| \\
		&\leq \delta \|\grad \bwl\|^2 + C(\delta, \L^{\kappa}_{\Bar{T}}) \big(\|\bwl\|^2 + 1\big)\;.
	\end{equs} 
	The other terms are rather easy to estimate, following similar arguments as in Section~\ref{sec:bound}\;. Thanks to the dissipative term $\bracket{\Laplace \bwl, \bwl} = - \|\grad \bwl\|^2$, finally we can get that 
	\begin{equs}
	\label{e:energy_estimate_HR23}
		\int_0^{\Bar{T}} \|\bwl[s]\|_{H^1}^2 \md s &\leq C(\lambda, \L^{\kappa}_{\Bar{T}})\;, \\
	\label{e:stopping_time_HR23}
		\Bar{T} &\geq c(\lambda, \L^{\kappa})\;
	\end{equs}
	with some constants $C(\lambda, \L^{\kappa}_{\Bar{T}})$ and $c(\lambda, \L^{\kappa})$. Then we estimate 
	\begin{equ}
		g_t := 1 + \bracket{\bwl[t], (-\Laplace)^{\kappa} \bwl[t]}\;.
	\end{equ}
	First, by Hölder and Sobolev inequality we have
	\begin{equs}
		\bracket{\P \div \big(\bwl\big)^{\otimes 2}, (-\Laplace)^{\kappa} \bwl } &\lesssim \|\grad \bwl\|_{H^{\kappa}} \|\big(\bwl\big)^{\otimes 2}\|_{H^{\kappa}} \\
		& \lesssim \|\grad \bwl\|_{H^{\kappa}}^2 + \|\big(\bwl\big)^{\otimes 2}\|_{H^{\kappa}}^2 \\
		& \lesssim \|\grad \bwl\|_{H^{\kappa}}^2 + \|\bwl \|_{H^{\kappa}}^2 \|\bwl\|_{H^1}^2\;.
	\end{equs}
	For the other terms appearing in $\frac{\md}{\md t}g_t $, we estimate them the same way as in Section~\ref{sec:energy_estimate} or in \cite[Lemma 5.3]{HR24}. Finally we get that 
	\begin{equ}
		\frac{\md}{\md t}g_t \leq C g_t \|\bwl\|_{H^1}^2 + C(\lambda, \L^{\kappa}_{\Bar{T}})\;.
	\end{equ}
	Combining it with \eqref{e:energy_estimate_HR23}, we get that for any $0 \leq s < t \leq \Bar{T}$, $g_t \leq C(\lambda, \L^{\kappa}_{\Bar{T}}) g_s$\,. Therefore, first taking an average in time and then using \eqref{e:energy_estimate_HR23} and \eqref{e:stopping_time_HR23}, we get 
	\begin{equ}
		g_{\Bar{T}} \leq \frac{1}{\Bar{T}}\int_0^{\Bar{T}}C(\lambda, \L^{\kappa}_{\Bar{T}}) g_s \md s \leq C_*(\lambda, \L^{\kappa}_{\Bar{T}})\;.
	\end{equ}
	Finally choose some large $C(\lambda)$ such that 
	\begin{equ}
		\PP(C_*(\lambda, \L^{\kappa}_{\Bar{T}}) > C(\lambda)) < \frac{1}{10}\;.
	\end{equ}
	Then on this event, we have
	\begin{equ}
		g_{\Bar{T}} = 1 + \|\bwl[\Bar{T}]\|_{H^{\kappa}}^2 \leq C(\lambda)\;.
	\end{equ}
	The proposition is proved.
\end{proof}

\begin{proof}[Proof of Proposition~\ref{pr:small_set}]
	We break the proof into several steps.
	\step{Step 1:} We first show that there exists some small time $t(K) \leq 1$ and some compact set $\fK \subset \cC^{-\kappa}$ both depending on $K$ such that 
	\begin{equ}
		\pP_{t(K)}(x, \fK) \geq \frac{3}{4}\;,\qquad \forall x\,:\, V_{2\alpha_0}(x) \le K\;.
	\end{equ}	
	Recall the decomposition $u = \Bar{X} + \Bar{Y} + w = \Bar{X} + \Bar{Y} + \bwh + \bwl$ previously defined.
	By \eqref{e:bwl_bound} and \eqref{e:bwh_bound} we get that there exists some event with probability larger than $9/10$, on which we have 
	\begin{equ}
		\|w[t_0(K)]\|_{H^{\kappa}} \leq C(K)\;,
	\end{equ}
	where $t_0(K) \leq 1$ is some small time depending on $K$. It follows that at time $t_0(K)$, with probability larger than $9/10$ we have $\|u[t_0(K)]\|_{\cC^{-1 + \frac{\kappa}{2}}} \leq C(K)$. Combining it with the fixed point argument in \cite{DPD02} (see also \cite[Proposition~3.2]{HR24}) we can show that at some time $t(K) > t_0(K)$, with probability larger than $3/4$ we have 
	\begin{equ}
		\|u[t(K)]\|_{\cC^{-\kappa/2}} \leq C'(K)\;.
	\end{equ}
	The claim then follows from the compactness of the embedding $\cC^{-\kappa/2} \hookrightarrow \cC^{-\kappa}$.

\step{Step 2:} We show that for any initial data $u_0$ with $V_{2\alpha_0}(u_0) \leq K$ and any small $\eps > 0$,
	\begin{equ}
		\PP\left(\|u[t_* - 1]\|_{\cC^{-\kappa}} \leq \eps\right) \geq \delta(\eps, K) > 0\;.
	\end{equ}
	By Proposition~\ref{pr:Strong_Feller}, the transition probability is continuous in total variation distance. Therefore,
	\begin{equ}
		\inf_{v\in \fK} \PP\left( \|\Phi_{t(K), t_* - 1} (v)\|_{\cC^{-\kappa}} \leq \eps \right)
	\end{equ}
	attains its minimum at some point since $\fK$ is compact. This is larger than $0$ by our support theorem \eqref{e:accessible}. 
		
\step{Step 3:} Now we start two SPDEs $u$ and $\Tilde{u}$ with independent noises to time $t_* -1$. Denote $\fA = \{\|u[t_* - 1]\|_{\cC^{-\kappa}} \leq \eps\,,\, \|\Tilde{u}[t_* - 1]\|_{\cC^{-\kappa}} \leq \eps \}$\,, then by independence and the previous step we have
	\begin{equ}
	\label{e:Doeblin1}
		\PP(\fA ) 
		\geq \big(\delta(\eps, K)\big)^2\;.
	\end{equ}
	Then we run the two processes from $t_* - 1$ to $t_*$ with the same noise. By continuity of total variation distance, for any $\eta > 0$, we can choose $\eps$ sufficiently small such that for any measurable $A \subset \cC^{-\kappa}$,
	\begin{equ}
	\label{e:Doeblin2}
		\sup_{\substack{\|x\|_{\cC^{-\kappa}}\leq \eps\\ \|y\|_{\cC^{-\kappa}}\leq \eps}} \Big| \PP \left(u[t_*] \in A; u[t_* - 1] = x\right) -  \PP \left(\Tilde{u}[t_*] \in A; \Tilde{u}[t_* - 1] = y\right)\Big| \leq 1- \eta\;.
	\end{equ} 
	Therefore, under this coupling, for any measurable $A\subset \cC^{-\kappa}$, we have
	\begin{equs}
		\,&\Big| \PP^{u_0} \left(u[t_*]\in A\right) - \PP^{\Tilde{u}_0} \left(\Tilde{u}[t_*]\in A\right) \Big| \\
		& \leq \Big|\PP^{u_0, \Tilde{u}_0}\left(u[t_*] \in A \,;\, \fA \right) - \PP^{u_0, \Tilde{u}_0}\left(u[t_*] \in A  \,;\, \fA \right)\Big|\\
		& + \Big|\PP^{u_0, \Tilde{u}_0}\left(u[t_*] \in A \,;\, \fA^c \right) - \PP^{u_0, \Tilde{u}_0}\left(u[t_*] \in A  \,;\, \fA^c \right)\Big|\;.
	\end{equs}
	The third line is always smaller than $\PP(\fA^c)$, since $|x - y| \leq x\vee y$ if $x, y \geq 0$. As for the second line, we first condition on $\fF_{t_*-1}$ and then use \eqref{e:Doeblin2} to get that it is smaller than $(1-\eta) \PP(\fA)$. Combining this with \eqref{e:Doeblin1}, the proposition is proved.	
\end{proof}

\begin{proof}[Proof of Theorem~\ref{thm:exp_mixing}]
	By Theorem~\ref{thm:main_result}, Proposition~\ref{pr:small_set} and \cite[Theorem~3.6]{Hai10} (see also \cite[Theorem~1.5]{HMS11} for a slightly different formulation), for any $k \in \N^*$ we have 
	\begin{equ}
		\|\pP_{k T_*}(x, \cdot) - \mu_{\star}\|_{TV} \leq C \rho^k (1 + V_{2\alpha_0}(x))
	\end{equ}
	with some $0 < \rho < 1$, provided that $T_*$ is taken to be large enough. This is a discrete time version of \eqref{e:exp_mixing}. The continuum version \eqref{e:exp_mixing} follows easily, since we can take $T_*$ to be any number that is large enough. For \eqref{e:exp_tail}, we start the equation from zero initial data. The result then follows by combining the exponential mixing result \eqref{e:exp_mixing}, moment bounds \eqref{e:main_result2} and Markov inequality.
\end{proof}

\bibliographystyle{Martin}
\bibliography{ref}
\end{document}